\documentclass[a4paper]{amsart} 

\usepackage{amsmath,amsthm,amssymb,amsfonts,mathrsfs,color,hyperref, mathtools,crop, graphicx, todonotes, enumitem}
\usepackage[]{geometry}

\theoremstyle{plain}
\begingroup
\newtheorem{theorem}{Theorem}[section]
\newtheorem*{theorem*}{Theorem}
\newtheorem*{"theorem"}{``Theorem''}
\newtheorem{corollary}[theorem]{Corollary}

\newtheorem{lemma}[theorem]{Lemma}
\endgroup

\theoremstyle{definition}
\begingroup
\newtheorem{definition}[theorem]{Definition}
\endgroup

\theoremstyle{remark}
\begingroup
\newtheorem{remark}[theorem]{Remark}
\newtheorem{example}[theorem]{Example}
\endgroup 

\numberwithin{equation}{section}
\newenvironment{pde}{\left\{\begin{array}{rll} } {\end{array}\right.}

\newcommand{\N}{\mathbb N}

\newcommand{\R}{\mathbb R} 

\newcommand{\dist}{{\rm dist}}

\renewcommand{\div}{{\rm div}}

\newcommand{\spt}{{\mathrm{spt}}}

\renewcommand{\H}{{\mathcal H}}
\newcommand{\E}{{\mathbb E}}

\newcommand{\M}{{\mathcal M}}

\newcommand{\LRa} {\Leftrightarrow}
\newcommand{\Ra} {\Rightarrow}

\renewcommand{\d}{\mathrm{d}}
\newcommand{\dx}{\,\mathrm{d}x}
\newcommand{\dy}{\,\mathrm{d}y}

\newcommand{\ds}{\,\mathrm{d}s}
\newcommand{\dt}{\,\mathrm{d}t}

\newcommand{\eps}{\varepsilon}
\newcommand{\average}{{\mathchoice {\kern1ex\vcenter{\hrule height.4pt
width 6pt depth0pt} \kern-9.7pt} {\kern1ex\vcenter{\hrule
height.4pt width 4.3pt depth0pt} \kern-7pt} {} {} }}

\renewcommand{\P}{\mathbb P}
\newcommand\showlabel{\addtocounter{equation}{1}\tag{\theequation}}
\DeclareMathOperator*{\argmin}{argmin} 
\newcommand{\Risk}{\mathcal{R}}
\renewcommand{\P}{\mathbb{P}}

 \allowdisplaybreaks

\makeatletter
\def\paragraph{\@startsection{paragraph}{4}%
  \z@\z@{-\fontdimen2\font}%
  {\normalfont\bfseries}}
\makeatother

 \makeatletter
\@namedef{subjclassname@2020}{2020 Mathematics Subject Classification}
\makeatother
 
\begin{document}

\title[Gradient Descent Training for Two-layer ReLU-networks]{On the Convergence of Gradient Descent Training for Two-layer ReLU-networks in the Mean Field Regime}

\author{Stephan Wojtowytsch}
\address{Stephan Wojtowytsch\\
Program in Applied and Computational Mathematics\\
Princeton University\\
Princeton, NJ 08544
}
\email{stephanw@princeton.edu}

\date{\today}


\begin{abstract}
We describe a necessary and sufficient condition for the convergence to minimum Bayes risk when training two-layer ReLU-networks by gradient descent in the mean field regime with omni-directional initial parameter distribution. This article extends recent results of Chizat and Bach to ReLU-activated networks and to the situation in which there are no parameters which exactly achieve MBR. The condition does not depend on the initalization of parameters and concerns only the weak convergence of the realization of the neural network, not its parameter distribution.
\end{abstract}

\maketitle


\section{Introduction}

Practitioners have found that artificial neural networks can be trained by gradient descent-based algorithms to fit many complicated target functions. While it is well-understood why the function class is sufficiently expressive for this purpose \cite{barron1993universal,cybenko1989approximation,hornik1991approximation,leshno1993multilayer}, the choice of optimal network parameters in applications is a highly non-convex problem. It is not fully understood why cleverly initialized gradient descent achieves good performance in practice.

In \cite{chizat2018global}, the authors prove that if the parameter distributions of neural network-like models converge to a limiting distribution, then the limit is in fact a global minimizer. The result is true for a class of `spread out' initial conditions (implying very wide networks) and under some technical assumptions. One of the most prominent cases concerns function models with the same homogeneity as neural networks with a single hidden layer and ReLU (or leaky ReLU) activation. However, the result does not apply directly due to the lack of differentiability of these activation functions.

In this article, we extend the main result of \cite{chizat2018global} for this ReLU-like setting and generalize previous results for some cases. The results proved here improve upon previous work in two ways:

\begin{enumerate}
\item Our analysis applies to ReLU-networks with suitable initial conditions rather than toy models with similar properties.
\item We only assume that a limiting object, whose existence is guaranteed by compactness, is unique. 
\end{enumerate}

In the analysis we exploit that ReLU activation
\[
\phi(a,w,b; x) = a\,\max\big\{w^Tx+b, 0\big\}
\] 
is not only positively two-homogeneous on parameter space $\R\times \R^d\times\R$, but in fact the product of two one-homogeneous functions (one of which is the identity map on $\R$). This product structure allows us to avoid certain `bad' points in the gradient flow. We state a special case of the main result informally.

\begin{theorem*}
Consider a two-layer mean field network model
\[
f_\pi(x) = \int_{\R\times\R^d\times \R} a\,\sigma(w^Tx+b)\,\pi(\d a\otimes\d w\otimes\d b)
\]
with activation function $\sigma = $ ReLU (or leaky ReLU) and parameter distribution $\pi$ on $\R\times \R^d\times \R$. Let $\pi_t$ evolve by the 2-Wasserstein gradient flow 
\[
\frac{d}{dt} \pi_t = \div \left(\pi_t\,\nabla \frac{\delta \Risk}{\delta\pi}\right)
\]
of a population risk functional
\[
\Risk(\pi) = \int_{\R^d\times\R } \ell\big(f_\pi(x), y\big)\,\P(\d x \otimes \d y)
\]
with a convex and Lipschitz-continuous loss function $\ell$. If the initial distribution $\pi_0$ is uniform on $[-1,1]\times S^d \subseteq \R\times \R^{d+1}$, then the velocity potential
\[
\frac{\delta \Risk}{\delta\pi}(\pi_t; a,w,b) = \int_{\R^d\times\R }(\partial_1\ell)\big(f_{\pi_t}(x),y\big)\,a\,\sigma(w^Tx+b)\,\P(\d x \otimes \d y)
\]
lies in a compact subset of $C^0_{loc}(\R^{d+2})$ as $t\to\infty$. The following are equivalent.

\begin{enumerate}
\item As time goes to infinity, $\Risk (\pi_t)$ converges to minimum Bayes risk.
\item The velocity potential $\frac{\delta \Risk}{\delta\pi}$ converges to a unique limit $g$ locally uniformly as $t\to\infty$.
\end{enumerate}
If $\Risk(\pi_t)$ approaches MBR, we can additionally identify the limit
\[
\lim_{t\to\infty} \frac{\delta \Risk}{\delta\pi} = 0.
\]
If additionally $\pi_t$ converges to a limit $\pi_\infty$ in 2-Wasserstein distance, $\pi_\infty$ minimizes $\R$.
\end{theorem*}

The study of Wasserstein gradient flows in the context of (infinitely wide) shallow neural networks is motivated below. The same problem exactly captures the gradient descent training of finite neural networks for infinitesimal learning rate and asymptotically captures stochastic gradient descent when the learning rate is so small/batch size is so large that stochastic effects are negligible.

A further technical condition has to be imposed on $\P$, and the result only holds conditionally on a Morse-Sard type property for suitable two-layer neural networks. The Sard-condition is purely technical, but has neither been verified nor disproved in general. In the appendix, we show that the condition holds in dimension $d=2$, but may fail in certain cases if $d\geq 8$. Without the Morse-Sard condition, we can show that the following are equivalent.

\begin{enumerate}
\item As time goes to infinity, $\Risk (\pi_t)$ converges to minimum Bayes risk.
\item The velocity potential $\frac{\delta \Risk}{\delta\pi}$ converges to $0$ locally uniformly as $t\to\infty$.
\end{enumerate}

The loss function is assumed to be $C^1$-smooth (with bounded derivative). We can consider more general loss functions (e.g.\ mean squared error), but only for uniformly bounded data. The initial condition can be generalized, but we require that at time $t=0$ we have $-a^2 + |w|^2+b^2\geq 0$ $\pi$-almost surely. The Minkowski inner product is preserved along the gradient evolution and the condition $(w,b)\neq 0$ unless also $a=0$ allows us to avoid issues of non-differentiability. To establish existence of a minimizer, it suffices that a certain projection of $\pi_t$ converges.

The significance of our result is as follows. 

\begin{itemize}
\item We do not need to assume the existence of a limit (guaranteed by compactness), but only its uniqueness.
\item One of the main complications of the question of convergence to minimal energy for this gradient flow is the dependence on the initial condition. The question whether the limit of the velocity potential is generically unique can be asked independently of the initial condition and seems more approachable by standard means of analysis. We therefore believe that this perspective might be a first step towards the convergence theory for Wasserstein gradient flows for shallow neural networks.
\end{itemize}

In the proof, we show directly that convergence to minimal energy implies convergence of the velocity potentials $\frac{\delta \Risk}{\delta\pi}$ to zero. This resembles the first order optimality condition in classic calculus and does not require deep geometric insight. In the converse direction, the geometry of the energy landscape is used via the homogeneity of the activation function. While we formulate all results for ReLU-activation, they hold equivalently for leaky ReLU networks. We prove that the second moments of the evolving parameter distribution $\pi_t$ grow at most sublinearly under general conditions and show that if the unique limit $g$ does not vanish identically, they grow exponentially. Thus only $g\equiv 0$ is admissible as a unique limit. In this situation the second moments -- which are bounded below by zero -- decrease linearly at a positive rate unless $\Risk(\pi_t)$ decays to minimum Bayes risk.

The article is structured as follows. In the remainder of this section, we briefly review previous work and collect some notation. In Section \ref{section background}, we describe the setting, state our main assumptions, and review some additional background information on Wasserstein gradient flows and continuity equations. In Section \ref{section main}, we state our main results. Some concrete examples of data distributions and loss functionals to which the results apply are listed in Section \ref{section examples}.
We conclude with a discussion of our results in Section \ref{section conclusion}. Longer proofs and some technical details are collected in Appendix \ref{appendix proofs}. Appendix \ref{appendix sard} is dedicated to the technical condition of Morse-Sard type which we assume.

\subsection{Related Work}
Introductions to machine learning in general and neural networks in specific can be found for example in \cite{shalev2014understanding,higham2019deep,goodfellow2016deep,mehta2019high}. 

The question why (stochastic) gradient descent starting at a suitable random initialization finds good parameters for artificial neural networks despite the fact that the energy landscape is highly non-convex has attracted much attention and several competing explanations have emerged. 

One avenue of research aims to uncover a description of gradient descent in neural networks in the infinite neuron limit in the mean field scaling regime \cite{chizat2018global,rotskoff2018neural, sirignano2018mean,mei2018mean}. Convergence criteria in the shallow network setting are developed for example in \cite{arbel2019maximum} and \cite{chizat2018global}, but have not been established generally.

Another line of articles \cite{weinan2019analysis,weinan2019comparative,bietti2019inductive, du2018gradient, du2018bgradient, jacot2018neural, arora2019exact} considers a heavily overparametrized regime with large initialization. In this setting, the gradient flow for neural networks behaves is proved to behave like the gradient flow of a very wide random feature model \cite{weinan2019comparative, weinan2019analysis} with high probability over the choice of (suitable) initial condition. In particular, the direction and bias of a neuron barely change from their (random) initialization in this regime. Chizat and Bach dub this the `lazy training regime' since neurons hardly move. They show that the underlying analysis is due rather to a (usually implicit) scaling assumption on initialization rather than the specific structure of neural networks \cite{chizat2018note}.

The success of neural networks in practical applications has been explained by the observation that -- unlike any linear theory -- neural networks can beat the `curse of dimensionality'  \cite{barron1993universal}. It thus seems unlikely that linearization is able to explain their recent success. Furthermore, in these studies parameters are usually initialized so large that the path norm a two-layer network with $m$ hidden neurons scales like $\sqrt{m}$ at initialization. Natural generalization bounds as derived in  \cite{weinan2019priori, E:2018ab} therefore do not apply.

Practitioners tend to train neural networks using stochastic gradient descent rather than full gradient descent. For small learning rate (time step size), the evolution can be described by SDEs with Gaussian noise \cite{hu2019diffusion,li2015dynamics}. The noise coefficient is given by the covariance of the gradient, which is usually neither isotropic nor homogeneous. However, in the mean field regime and assuming that the noise is standard Gaussian, one can prove that the parameter distribution converges to a good value (minimizer of a regularized risk functional), see \cite{hu2019mean}. The derivation is built on the link of the heat equation to both stochastic analysis and optimal transport theory \cite{jordan1998variational}. In this case, the parameter distribution approaches the stationary measure of a Markov process as time approaches infinity. If the noise is sufficiently large, the convergence is exponential, while for small noise, the stationary measure approaches a minimizer of the mean field risk functional.

Rigorous convergence results in realistic settings can be obtained under strong assumptions on the initial condition (or a state which arises along the gradient flow). One example is \cite{berlyand2018convergence}, where the authors show that if a neural network has well chosen parameters at some time, the parameters improve along the gradient flow. In \cite{arbel2019maximum} the authors study a mean field gradient flow and show that under a regularity/closeness assumption, the gradient flow converges. Unfortunately, the condition cannot be verified in practice. Global convergence for a toy model with similar properties is established in \cite[Section 7]{E:2019aa}.

Some results in \cite{chizat2018global} also apply to neural networks with smooth activation and more than one hidden layer. However, the imposition of a linear structure results in network-like models where each neuron in the outermost layer has its own trainable weights for the deeper layers. 
More recent works in the mean field setting consider deep neural networks whose parameters are initialized independently across the layers \cite{araujo2019mean,nguyen2020rigorous, sirignano2019mean}. The independence is preserved through time (`propagation of chaos') and a mean field description is available in different scaling limits. The theory is entirely different from that of shallow networks. While a shallow networks can be described by indexed particles $(a_i, w_i, b_i)$, the paths in a deep network through multiple layers have a more complicated interacting multi-index structure $(a_i, b_{ij}, c_j)$.

\subsection{Notations and Terminology}

We denote by $|\cdot|$ the Euclidean norm on any finite-dimensional vector space. For a map $\phi:A\to B$ and a measure $\mu$ on $A$, we denote the push-forward of $\mu$ along $\phi$ by $\phi_\sharp \mu$ (which is a measure on $B$), see e.g.\ \cite{evans2015measure}. By $\spt(\mu)$, we denote the support of the measure $\mu$. For a measure $\mu$ and a locally $\mu$-integrable function $f$ we denote by $f\cdot\mu$ the measure
\[
(f\cdot\mu)(B) = \int_Bf\d\mu
\]
which has density $f$ with respect to $\mu$ (occasionally denoted by $\mu|_f$ in other texts). We denote probability measures on data space by $\P$, projection operators by $P$ and the space of probability measures by $\mathcal P$. Parameter distributions are denoted by $\pi$ and assumed to be elements of the space $\mathcal P = \mathcal P_2$ of probability measures with finite second moments (Wasserstein space).

\section{Background and Assumptions}\label{section background}

\subsection{Activation, Loss and Data}

In this section, we describe the objects under consideration in this article. Some assumptions will be relaxed in Section \ref{section main general loss}. 

\paragraph{Activation function.} We consider the {\em parameter domain}
\[
\Theta = \big\{(a,w,b)\in\R\times\R^d\times \R\::\: a^2 < |w|_{\ell^2}^2 + b^2\}
\]
and the {\em activation function} $\phi:\overline\Theta\to\R$
\[
\phi(\theta; x) = a\,\sigma(w^Tx+b)\qquad\text{where} \quad \theta = (a,w,b), \quad \sigma(z) = z_+ = \max\{z,0\}.
\]
For this introduction, the choice of cone $\Theta$ is inessential and will be motivated later in Section \ref{section initial}. Note  that the family
\[
\{\phi(\theta; x) : \theta \in \Theta\}
\]
has the universal approximation property, i.e.\ for any continuous function $f$ on a compact set $K\subseteq \R^d$ and every $\eps>0$, there exists a finite collection of parameters $\{\theta_1,\dots,\theta_m\}$ such that
\begin{equation}\label{eq UAT}
\left\| f - \frac1m\sum_{i=1}^m \phi(\theta_i,\cdot)\right\|_{C^0(K)} < \eps.
\end{equation}
The property is known to hold when the parameters $\theta_i$ are not constrained to a cone, the factor $\frac1m$ is not present and coefficients $a_i$ are included before $\phi(\theta_i,\cdot)$ \cite{cybenko1989approximation}. None of these differences are significant as
\[
\phi\left(\lambda a; \frac{w}\lambda, \frac{b}\lambda; \cdot \right) \equiv \phi(a,w,b; \cdot) \quad \forall\ \lambda>0\qquad
\text{and}\quad \lambda \,\phi(a,w,b;\cdot) \equiv \phi(\lambda a,w,b;\cdot) \quad \forall\ \lambda\in \R.
\]

\paragraph{Parameter distribution.} A {\em parameter distribution} $\pi$ is a Borel probability measure on $\overline\Theta$ with finite second moments
\[
N(\pi):= \int_{\overline\Theta}|\theta|^2\,\pi(\d\theta) = \int_{\overline\Theta} |a|^2 + |w|_{\ell^2}^2 + |b|^2\,\pi(\d a\otimes \d w \otimes \d b),
\]
i.e.\ $\pi \in \mathcal{P}:= \mathcal{P}_2(\overline\Theta)$ is an element of Wasserstein space over the cone $\overline\Theta$. We denote the {\em realization} of $\pi$ as
\[
f_\pi(x) = \int_{\overline\Theta} \phi(\theta; x)\,\pi(\d \theta)
\]
and note that $f_\pi$ is $N(\pi)$-Lipschitz and defined on the whole space $\R^d$.

\paragraph{Loss function.} We further assume that the {\em loss function} $\ell:\R\times\R \to [0,\infty)$ satisfies the following.
\begin{enumerate}[label=(L\arabic*), ref=(L\arabic*)]
\item\label{assumption continuity l} $\ell$ is jointly continuous.
\item\label{growth condition on l} $\ell$ is once continuously differentiable in the first argument and there exists $C_\ell>0$ such that
\[
|\partial_1\ell| (y,y') \leq C_\ell \qquad\forall\ y,y'\in \R.
\]
\item\label{assumption convexity l} The function $y\mapsto \ell(y,\bar y)$ is convex for all $\bar y\in \R$.
\end{enumerate}

Except for the Lipschitz condition \ref{growth condition on l}, the assumptions are common and non-restrictive. The assumption that $|\partial_1\ell|$ is uniformly bounded is required for technical purposes to compensate the singularity of the ReLU function. It will be relaxed in Section \ref{section main general loss}. 

\paragraph{Data distribution.}
Finally, we consider {\em data-distributions} $\P$ on $\R^d\times \R $ such that

\begin{enumerate}[label=(P\arabic*), ref=(P\arabic*)]
\item\label{assumption P Borel} $\P$ is a Borel measure, i.e.\ every continuous function on $\R^d\times \R $ is $\P$-measurable.
\item\label{assumption P finite moments} The first moments
\[
\int_{\R^d\times \R } |x| + |y|\,\P(\d x\otimes \d y)
\]
are finite.
\item\label{assumption density in L1} The space spanned by the collection of functions $f(\theta_i;\cdot)$ is dense in $L^1(\overline\P)$ where $\overline\P:= P^x_\sharp \P$ denotes the projection of $\P$ on the first component.
\item\label{assumption RL1} We assume that the map
\[
S^d\to L^1\big(\sqrt{1+|x|^2}\cdot\overline\P\big), \qquad (w,b) \mapsto 1_{\{x|w^Tx+b>0\}}
\]
is Lipschitz continuous.
\end{enumerate}

\ref{assumption P Borel} is a non-restrictive technical condition. \ref{assumption P finite moments} is required to show that the risk of linearly growing functions is always finite. \ref{assumption density in L1} is always met if $\overline\P$ is compactly supported due to the Universal Approximation Theorem \cite{cybenko1989approximation} and the density of continuous functions in $L^1$ for any Radon measure \cite[Theorem 2.11]{fonseca2007modern}. \ref{assumption RL1} imposes a high degree of smoothness on the data distribution which is used to compensate for the lack of differentiability of the activation function $\phi$ at points $(a,w,b)$ and $x$ for which $w^Tx+b=0$. Intuitively, $\overline\P$ cannot concentrate on or close to lower-dimensional {\em linear} objects. A full discussion of admissible data distributions is beyond the scope of this article, but we give some examples.

\begin{theorem}\label{theorem admissible data}
\begin{enumerate}
\item Assume that $\overline\P$ has a density $\rho$ with respect to Lebesgue measure on $\R^d$ such that \[
\rho(x) \leq C{\big(1+ |x|^2\big)^{-\frac{d+2+\eps}2}}
\]
for some $C,\eps>0$. Then $\overline\P$ satisfies \ref{assumption RL1}.

\item The uniform distribution on the unit sphere satisfies \ref{assumption RL1}.

\item The set of data-distributions $\overline\P$ satisfying \ref{assumption RL1} is convex.

\item For any $L>0$ the set of data-distributions $\overline\P$ satisfying
\[
\big\|1_{\{w^Tx+b>0\}} - 1_{\tilde w^Tx+\tilde b>0\}}\big\|_{L^1(\sqrt{1+|x|^2}\cdot \overline\P)} \leq L\big[|w-\tilde w| + |b-\tilde b|\big]
\]
is convex and closed under weak convergence of Radon measures on $\R^d$.
\end{enumerate}
\end{theorem}

We note that as a finite Borel measure on the locally compact Polish space $\R^d$, $\P$ is in fact a Radon measure. Due to \cite[Theorem 4.2.4]{attouch2014variational} $\P$ can be decomposed into conditional probabilities $\P^x$ and a distribution $\overline\P = P^x_\sharp \P$ on $\R^d$ like in \ref{assumption  density in L1} such that
\begin{enumerate}
\item for any $\P$-measurable function $g:\R^d\times\R \to\R$, the map
\[
x\mapsto \int_{\R } g(x,y)\, \P^x(\d y)
\]
is $\overline\P$-measurable and
\item the equality
\[
\int_{\R^d\times\R } g(x,y)\,\P(\d x \otimes \d y)= \int_{\R^d}\left(\int_{\R } g(x,y)\,\P^x(\d y)\right) \,\overline\P(\d x)
\]
holds for all $\P$-measurable functions $g$.
\end{enumerate} 
Consider the {\em augmented loss function}
\[
L: \R^d \times \R\to [0,\infty), \qquad L_x(\alpha):= \int_{\R } \ell(\alpha,y)\,\P^x(\d y)
\]
which encodes many important properties of the problem. 

\subsection{The Risk Functional}\label{section risk functional}
Combining all previous notions, we define the risk functional $\Risk:\mathcal P_2(\overline \Theta) \to [0,\infty]$,
\begin{equation}
\Risk(\pi):= \int_{\R^d\times\R } \ell\big(f_\pi(x), y\big)\,\P(\d x \otimes \d y) = \int_U L_x\big(f_\pi(x)\big) \,\overline\P(\d x).
\end{equation}

We impose the following compatibility conditions between $\ell$ and $\P$ to control the behaviour of $\Risk$.

\begin{enumerate}[label=(LP\arabic*), ref=(LP\arabic*)]\setcounter{enumi}{0}
\item\label{assumption growth at infinity} {\em Growth of augmented loss}:
\[
\lim_{|\alpha|\to\infty} L_x(\alpha) = \infty
\]
for $\overline\P$-almost every $x\in \R^d$.

\item {\em Finite risk}:\label{assumption finite risk}
\[
\int_{\R^d\times \R }\ell(0, y)\,\P(\d x\otimes \d y) <\infty.
\]
\end{enumerate}

\begin{remark}
A few observations are in order. Recall that $\mathcal P$ or $\mathcal P_2$ denotes the space of Radon measures with finite second moments (on the natural space in a given context, which usually is $\overline\Theta$).
\begin{enumerate}
\item Since $\ell$ and $f_\pi$ are Lipschitz continuous, we can use \ref{assumption finite risk} and \ref{assumption P finite moments} to deduce that $\Risk(\pi)<\infty$ for all $\pi \in\mathcal P_2$.
\item For any fixed $x$, the function $L_x$ is convex and differentiable in $\alpha$ if it is finite (i.e.\ for almost every $x$). If $\ell$ is strictly convex, so is $L_x$.
\item The function $(x,\alpha)\to L_x(\alpha)$ is measurable (see Lemma \ref{lemma measurable} below), but cannot generally be assumed to be continuous.
\end{enumerate}
\end{remark}

We can extend the risk functional to the larger space of $\overline\P$-measurable functions by setting
\[
\widetilde \Risk (f) := \int_{\R^d\times\R } \ell\big(f(x), y\big)\,\P(\d x\otimes\d y) = \int_U L_x\big(f(x)\big) \,\overline\P(\d x)
\]
with possibly infinite values. $L_x$ is convex and continuous in $\alpha$ for $\overline\P$-almost every $x$ and satisfies the growth condition \ref{assumption growth at infinity}, so there exists a compact convex set $M_x\subset \R $ such that $L_x(\alpha) = \inf_{\alpha'\in\R } L_x(\alpha')$ for all $\alpha\in M_x$. If $L_x$ ist strictly convex (for example because $\ell$ is strictly convex), then the minimum of $L_x$ is unique. 

We note that there exists a measurable selection of minimizer for $L_x$. This is not entirely immediate even when $L_x$ is strictly convex. 

\begin{lemma}\label{lemma measurable}
\begin{enumerate}
\item The function $(x,\alpha)\to L_x(\alpha)$ is measurable on $\R^d\times\R $ with respect to the product $\sigma$-algebra generated by $\overline\P$ and the Borel sigma algebra on $\R $.

\item There exists a $\overline\P$-measurable function $f^*:\R^d\to\R $ such that $f^*(x)\in M_x$ for $\overline \P$-almost every $x\in \R^d$. $f^*$ satisfies
\[
(\nabla L_x)\big(f^*(x)\big) = \int_{\R } \big(\partial_1\ell\big)\big(f^*(x),y\big)\,\P^x(\dy) = 0\qquad\overline\P-\text{a.e.}
\]
\end{enumerate}
\end{lemma}

It is clear that
\[
\Risk(\pi) = \widetilde \Risk (f_\pi) \geq \widetilde \Risk (f^*)
\]
and that equality is achieved if and only if $f_\pi(x)\in M_x$ $\overline\P$-almost everywhere. In particular, if $L_x$ is strictly convex, equality holds if and only if $f_\pi= f^*$ $\overline\P$-almost everywhere. We add the following compatibility condition between $\P$ and $\ell$.

\begin{enumerate}[label=(LP\arabic*), ref=(LP\arabic*)]\setcounter{enumi}{2}
\item\label{assumption minimizer L1} {\em Minimum Bayes risk}: There exists a version of $f^*$ in $L^1(\overline\P)$. 
\end{enumerate}

Since $\ell$ is Lipschitz-continuous in the first argument, assumption \ref{assumption density in L1} implies that
\[
\inf_{\pi \in \mathcal P_2(\overline\Theta)} \Risk(\pi) = \widetilde \Risk(f^*)
\]
as $f^*$ can be approximated arbitrarily well in $L^1(\overline\P)$ by functions of type $f_\pi$. Thus $\widetilde\Risk(f^*)$ is the minimum Bayes risk of the problem.

\begin{corollary}
The functional $\Risk$ admits a minimizer if and only if there exists a measure $\pi\in\mathcal P$ such that
\[
f_\pi = f^*\quad \overline\P-\text{almost everywhere.}
\]
\end{corollary}

In particular, as $f_\pi$ is $N(\pi)$-Lipschitz (where again $N(\pi)$ denotes the second moments of $\pi$), there must be a version of $f^*$ which is Lipschitz-continuous. The Lipschitz condition is far from sufficient \cite{approximationarticle}. Since 
\[
f_{\lambda\pi + (1-\lambda)\pi'} = \lambda f_\pi + (1-\lambda)\,f_{\pi'} \qquad\forall\ \pi,\pi' \in \mathcal P, \:\lambda\in[0,1],
\]
the set of minimizers is a convex subset of $\mathcal P$.

\begin{remark}
Write $T(a,w,b) = (-a, w,b)$. The fact that $\phi(T\theta,\cdot) \equiv -\phi(\theta,\cdot)$ implies that the measure $\pi$ representing a function $f_\pi$ cannot be unique: Any measure $\pi$ such that $T_\sharp \pi = \pi$ (i.e.\ $\pi(T^{-1}(V)) = \pi(V)$ for all measurable $V\subseteq \Theta$) represents the function $0$ since
\[
f_\pi(x) = \int \phi(\theta,x) \,\pi(\d\theta) = \int\phi(\theta,x) \,T_\sharp\pi(\d\theta) = \int \phi(T\theta,x)\,\pi(\d\theta) =- \int \phi(\theta,x) \,\pi(\d\theta) = - f_\pi(x).
\]
In ReLU networks, another source of non-uniqueness is the identity
\[
0 = (x+1) - x- 1 = \sigma(x+1) - \sigma\big(-(x+1)\big) - \sigma(x) + \sigma(-x) - \sigma(1).
\]
More generally, if $\pi$ represents $f_\pi$ and $\pi'$ represents $0$, then for any $\lambda\in (0,1)$ consider the probability measure
\[
\pi_\lambda = \lambda\, D(\lambda)_\sharp \pi+ (1-\lambda)\pi'
\]
where $D(\lambda) = \mathrm{diag}(\lambda^{-1/2})$ is a dilation. Then $\pi_\lambda$ is a probability measure with finite second moments and
\[
f_{\pi_\lambda} = \lambda f_{D(\lambda)_\sharp\pi} + (1-\lambda)\,f_{\pi'} = 
\lambda \int_\Theta f(\lambda^{-1/2}\theta,\cdot) \,\pi(\d\theta) = \int_\Theta f(\theta,\cdot)\,\pi(\d\theta) = f_\pi.
\]
In particular, while the minimizer $f^*$ of $\widetilde \Risk $ is unique if $L_x$ is strictly convex, a minimizer of $\Risk$ must be highly non-unique. This is a major obstacle in linearization-based approaches to convergence.
\end{remark}

\subsection{Wasserstein Gradient Flows}\label{section wasserstein gf}
The study of Wasserstein gradient flows in machine learning is motivated by the following observation: {\em The parameters $\{\theta_i\}_{i=1}^m \in \overline\Theta^m$ of a parametrized function
\[
f_{\theta_1,\dots\theta_m}(x) = \frac1m \sum_{i=1}^m \phi(x,\theta_i)
\]
evolve by the time-accelerated Euclidean gradient flow
\[
\dot \theta_i = -m\,\nabla_{\theta_i} \Risk\big(\theta_1(t),\dots,\theta_m(t)\big)
\]
of 
\[
\Risk(\theta_1,\dots,\theta_n) = \int_{\R^d\times \R} \ell\big(f_{\theta_1,\dots,\theta_m}(x), y\big)\,\P(\d x\otimes \d y)
\]
if and only if their distribution $\pi_m = \frac1m \sum_{i=1}^m \delta_{\theta_i}$ follows the $2$-Wasserstein gradient flow of the extended risk functional
\[
\Risk(\pi) = \int_{\R^d\times \R} \ell\big(f_\pi(x), y\big)\,\P(\d x\otimes \d y)
\]
where
\[
f_\pi(x) = \int_{\overline\Theta} \phi(\theta;x)\,\pi(\d\theta).
\]
}

A key observation is that the individual particles $\theta_i$ are irrelevant and only their distribution matters when computing $f_{\theta_1,\dots,\theta_m}$. We refer to two-layer network functions of the form $f(x) = \frac1m\sum_{i=1}^m a_i\,\sigma(w_i^Tx+b_i)$ as {\em mean field networks} in contrast to {\em classical two-layer networks} $f(x) = \sum_{i=1}^m a_i\,\sigma(w_i^Tx+b_i)$. Both classes are identical from the perspective of approximation theory, but lead to different dynamic models in the infinite-width limit. Classical networks are described by the linearized dynamics of neural tangent kernels, while mean field networks evolve truly non-linearly by Wasserstein gradient flows.

Thus optimizing mean field network parameters by the gradient flow of a risk functional is equivalent to optimizing their distribution by a Wasserstein gradient flow, see e.g. \cite[Proposition B.1]{chizat2018global}. An expanded heuristic can also be found in Appendix \ref{appendix proofs}. 

The gradient flow and the map $\pi\mapsto f_\pi$ (so also the risk functional $\Risk$) are naturally defined on the Wasserstein space $\mathcal P_2$ of probability measures with finite second moments. Background on optimal transport theory and Wasserstein gradient flows can be found e.g.\ in \cite{ambrosio2008gradient, santambrogio2015optimal, villani2008optimal}.

In this article, we mostly consider Wasserstein gradient flows of continuous distributions. By continuity, the Wasserstein gradient flows $\pi_m(t)$ starting at $\pi_m^0$ converge to the solution of the Wasserstein gradient flow $\pi(t)$ starting at $\pi^0 = \lim_{m\to\infty}\pi_m^0$ for all $t>0$. Even more, the limits
\[
\lim_{m\to\infty}\lim_{t\to\infty}\pi_m(t) = \lim_{t\to\infty}\lim_{m\to\infty}\pi_m(t) = \lim_{t\to\infty} \pi(t)
\]
commute (if they exist). A proof under an additional technical condition (which can be eliminated if one considers gradient flows of unregularized risk or the second moment regularizer) is given in Appendix B of \cite{chizat2018global}. The complication arises from regularizing functionals which do not have the correct homogeneity, which leads the authors to consider compactly supported initial conditions.

Hence if the gradient flow starting at a measure $\pi^0$ converges to a minimizer of risk, then gradient flows starting at closeby empirical measures asymptotically achieve low risk. The theory, at this point, is purely qualitative, but nonetheless indicative for practical applications.

The Wasserstein gradient flow $\pi_t$ is described by the continuity equation
\begin{equation}\label{eq gradient flow}
\frac{d}{dt} \pi_t= \div_\theta\big(\pi_t\,\nabla_\theta(\delta_\pi\Risk )(\pi_t,\cdot)\big)
\end{equation}
in the distributional sense, i.e.\
\[
\frac{d}{dt} \int_\Theta g(\theta)\,\pi_t(\d\theta) = - \int_\Theta \big\langle\nabla g(\theta), \nabla_\theta(\delta_\pi\Risk )\big\rangle\,\pi_t(\d\theta)
\]
for $g\in C_c^\infty(\Theta)$ where
\begin{align*}
\delta_\pi\Risk (\pi,\theta) &= \int_{\R^d\times\R } (\partial_1\ell)(f_\pi(x),y) \, \phi(\theta,x)\,\P(\d x \otimes \d y)\\
\nabla_\theta \,\delta_\pi\Risk (\pi,\theta) &= \int_{\R^d\times\R } (\partial_1\ell)(f_\pi(x),y) \,\nabla_\theta \phi(\theta,x)\,\P(\d x \otimes \d y)
\end{align*}
denote the variational derivative of $\Risk$ with respect to $\pi$ and its spatial gradient respectively.

\subsection{Continuity Equations}\label{section continuity eqn}
Consider a general continuity equation 
\[
\begin{pde}
\frac{d\mu}{dt} &= \div\left(\mu V\right)&t>0\\
\mu &= \mu_0(\theta) & t=0
\end{pde}
\]
in the space of finite Radon measures on the whole space $\R^d$ and the associated flow map
\begin{equation}\label{eq flow map}
\begin{pde}
\frac{d X}{dt}(t,\theta) &= V\big(t, X(t,\theta)\big)&t>0\\
X(0,\theta) &= \theta
\end{pde}.
\end{equation}
Assume for the moment that $V$ lies in the Bochner space $L^1\big((0,T), W^{1,\infty}(\R^d;\R^d)\big)$, i.e.\ $V$ is Lipschitz-continuous in space at almost every time $t>0$ and 
\[
\int_0^T \|V(t,\cdot)\|_{L^\infty(\R^d)}+ \int_0^T\|\nabla_\theta V(t,\cdot)\|_{L^\infty(\R^d)}\dt <\infty.
\]
Then according to \cite[Proposition 4]{ambrosio2008transport}, we have
\begin{equation}\label{eq flow}
\mu_t = X(t,\cdot)_\sharp \mu_0,\quad\text{i.e. }\int_{\R^d}g(\theta)\,\mu_t(\d\theta) = \int_{\R^d}g\big(X(t,\theta)\big)\,\mu_0(\d\theta)
\end{equation}
for all $g \in C_c^\infty(\R^d)$. Note that by Gr\"onwall's Lemma we also have
\begin{equation}\label{eq lipschitz bound}
\mathrm{Lip}(X(t,\cdot))\leq \exp\left(\int_0^t \mathrm{Lip}(V(s,\cdot))\ds\right).
\end{equation}
The Lemma does not apply directly to the situation which we will consider since the flow field $V$ will be positively one-homogeneous and thus unbounded in most cases. However, the result also applies under the weaker assumption
\[
\frac{V(t,\theta)}{1+ |\theta|} \in L^1\big([0,T],\,L^\infty(\R^d;\R^d)\big),
\] 
see \cite[Remark 7]{ambrosio2008transport}. The condition of at most linear growth is required to prevent particles from escaping to infinity in finite time. Thus \eqref{eq flow} and \eqref{eq lipschitz bound} apply also in our situation.

\subsection{Initial Condition}\label{section initial}

There are two considerations concerning the initial parameter distribution of the gradient flow. The first is specific to ReLU activation and of a technical nature while the second one is more geometric and concerns energy decay to minimum Bayes risk/convergence to minimizers.

\subsubsection{The cone of good parameters}
We can formally compute the parameter gradient of the activation function
\[
\nabla_\theta \phi(\theta; x) = \begin{pmatrix} \sigma(w^Tx+b)\\ a\,\sigma'(w^Tx+b)\,x\\ a\,\sigma'(w^Tx+b)\end{pmatrix}
\]
which is defined at all $\theta = (a,w,b)$ and $x$ such that $w^Tx+b\neq 0$ with $\sigma'(z) = 1_{\{z>0\}}$. In particular, if for every $(w,b)\neq 0$ the hyperplane $\{x:w^Tx+b=0\}$ is a $\overline\P$-null set, then the parameter gradient of risk
\[
\nabla (\delta_\pi \Risk)(\pi; \theta) = \int_{\R^d\times \R} (\partial_1\ell)\big(f_\pi(x), y\big)\,(\nabla_\theta \phi)(\theta;x)\,\P(\d x\otimes \d y) 
\]
is well-defined at $(w,b)\neq 0$. This is a regularity condition on $\overline \P$ which is satisfied for example whenever $\overline\P$ has a density with respect to Lebesgue measure or the uniform distribution on a sphere. For technical reasons, we need the gradient to even be Lipschitz continuous, which is why we imposed assumption \ref{assumption RL1} on the data distribution $\P$. 

There is a more subtle problem of regularity. Assumption \ref{assumption RL1} guarantees that 
\[
(w,b) \mapsto \int_{\R^d\times\R} (\partial_1\ell) \big(f_\pi(x), y\big)\,1_{\{w^Tx+b\}}\,\P(\d x\otimes \d y)
\]
is a Lipschitz-function on the sphere $S^d\subseteq\R^{d+1}$. However, close to a point where $(w,b) = 0$ and $a\neq 0$, the half-spaces $\{x:w^Tx+b>0\}$ oscillate rapidly and the gradient
\[
(a,w,b) \mapsto \int_{\R^d\times\R} (\partial_1\ell) \big(f_\pi(x), y\big)\,1_{\{w^Tx+b\}}\,\P(\d x\otimes \d y)
\]
generally fails to even be continuous. This difficulty can only be circumvented if $a$ is close to zero whenever $(w,b)$ is close to zero. At this point, the geometry of the ReLU function as the product of two positively one-homogeneous functions comes into play. Namely, we can use the fact that 
\[
\sigma(z) = \sigma'(z)z
\]
to show that along trajectories of the flow map $X$ in \eqref{eq flow map}, the Minkowski norm $-a^2 + |w|_{\ell^2}^2 + b^2$ remains constant: 
\begin{align*}
\frac d{dt}& \big[-a^2 + |w|_{\ell^2}^2 + b^2\big] = 2 \big[-a\dot a + \langle w, \dot w\rangle + b\dot b\big]\\
	&= 2\big[a\,\partial_a(\delta_\pi \Risk) - \langle w, \nabla_w(\delta_\pi\Risk)\rangle -b\,\partial_b(\delta_\pi\Risk)\big]\\
	&= 2\int_{\R^d} (\partial_1\ell)(f_\pi(x), y) \big[a\,\sigma(w^Tx+b) - a\,\sigma'(w^Tx+b)\cdot(w^Tx+b)\big]\,\P(\d x \otimes \d y)\\
	&= 0.
\end{align*}
In particular, the cone of space-like vectors 
\[
\Theta = \{(a,w,b): -a^2 + |w|^2 + b^2>0\}
\]
is preserved along the gradient flow evolution. Inside $\overline\Theta$, we have $|a| \leq \sqrt{|w|^2 + b^2}$, which allows us to save the Lipschitz property. We therefore restate an assumption on the initial condition more explicitly.

\begin{enumerate}[label=(IC\arabic*), ref=(IC\arabic*)]\setcounter{enumi}{0}
\item\label{assumption initial cone} {\em Small linear variable}: $\pi_0\in \mathcal{P}_2(\overline\Theta)$, i.e.\ $\pi_0$ is a probability measure on $\R^{d+2}$ with finite second moments such that $\spt(\pi_0)\subseteq \overline\Theta$.
\end{enumerate}

\begin{remark}
It is possible to consider initial conditions supported on other super-level sets like $\Theta_\eps = \{(a,w,b): -a^2 + |w|^2 + b^2>\eps^2\}$ instead. A flow starting at $\pi_0$ supported on $\Theta_\eps$ will never reach a point where $(w,b) = 0$, but on the other hand does not allow us to exploit the positive two-homogeneity of $\phi$ in $\theta$ as easily since $\Theta_\eps$ is not a cone. Positive two-homogeneity is essential to many arguments below, so we proceed with $\Theta = \Theta_0$ instead.
\end{remark}

\subsubsection{Omni-directional initial conditions}

It is well-known that the functional $\Risk$ is not sufficiently convex in Wasserstein geometry to guarantee convergence to global minimizers from any initial condition. In particular, if a global minimizer $\pi$ exists but cannot be written as an empirical measure with $m$ atoms, any initial condition corresponding to an empirical measure with $m$ atoms {\em cannot} converge to $\pi$ since the continuity equation has no smoothing effect and preserves atomic measures. 

The following class of initial conditions is successful in theory and applications.

\begin{definition}[Omni-directional measure]
We call a probability measure $\pi$ on $\Theta$ {\em omni-directional} if every open cone in $\Theta$ has positive measure.
\end{definition}

\begin{enumerate}[label=(IC\arabic*), ref=(IC\arabic*)]\setcounter{enumi}{1}
\item\label{assumption omnidirectional} {\em Omnidirectional initial condition}: $\pi_0$ is omni-directional.
\end{enumerate}

Clearly, an omni-directional initial condition is an abstraction available only in the infinite-width limit.

\begin{remark}
A measure $\pi_0$ can be omni-directional in $\Theta= \Theta_0$ and yet supported on a smaller set $\Theta_\eps$ for $\eps>0$ since for any $\theta\in \Theta$ and $\lambda\in\R$ the Minkowski norm of $\lambda\theta$ is
\[
\lambda^2(-a^2 + |w|^2+b^2) >\eps^2 \quad \forall\ \lambda> \frac{\eps}{\sqrt{-a^2+|w|^2 + b^2}}.
\]
Thus for any $\theta\in \Theta$, there exists $\lambda>0$ such that $\lambda\theta\in \Theta_\eps$. In particular, if $\spt(\pi_0) = \overline{\Theta_\eps}$ for any $\eps>0$, then $\pi_0$ is omni-directional.
\end{remark}

\subsection{Morse-Sard Property}
Below, we prove existence and uniqueness for gradient flow training of ReLU-activated two-layer networks under the assumptions listed above. For technical reasons, we add an assumption which has only been established in full generality in dimension $d=2$. The assumption is only required when discussing the limiting behavior of the gradient flow.

 \begin{enumerate}[label=(M-S), ref=(M-S)]\setcounter{enumi}{0}
\item\label{assumption Sard} {\em Morse-Sard condition}: Assume that 
\[
(\delta_\pi \Risk)(\pi_{t_n}, \cdot) \to g,\qquad \nabla (\delta_\pi \Risk)(\pi_{t_n}, \cdot) \to V
\]
locally uniformly on $\overline\Theta$. Then the restriction of $g$ to the unit sphere $S^{d+1} = \{a^2 + |w|^2 + b^2 =1\}$ has the property that $P^{S^d}_*V= \nabla^{S} g$ does not vanish anywhere on the level set $\{g=t\}$ for Lebesgue-almost every $t$, where $\nabla^S$ denotes the gradient of $g$ tangent to the sphere and $P^{S^d}_*$ the tangent map to the projection onto the sphere.
\end{enumerate}

We discuss the Morse-Sard condition below in Appendix \ref{appendix sard}. Using smoothness, we establish the property in dimension two and provide evidence on the other hand that \ref{assumption Sard} may not be expected to hold in high dimension in full generality. We isolate the point in the proof where the condition is used and where an argument would need to be adapted in order to avoid it.  

A weaker version of the main result holds without this condition.

\section{Evolution of the Parameter Distribution}\label{section main}

\subsection{Existence and Uniqueness}

We first establish that solutions to gradient flow training exist. Assume all conditions outlined above except for \ref{assumption omnidirectional} and \ref{assumption Sard}, which are only required for statements about limiting objects. 

\begin{lemma}
\label{lemma existence}
Let $\pi_0\in \mathcal P_2(\overline\Theta)$. Then there exists a unique solution to the Wasserstein gradient flow \eqref{eq gradient flow}.
\end{lemma}

The next statement will be useful to understand the behavior of limits as $t\to\infty$.

\begin{lemma}\label{lemma omni-directional}
If $\pi_0 \in \mathcal{P}_2(\overline\Theta)$ is omni-directional and $\pi_t$ denotes the Wasserstein gradient flow starting at $\pi_0$, then $\pi_t$ is omni-directional for all $t>0$.
\end{lemma}

This Lemma is a simpler version of \cite[Theorem 3.3]{chizat2018global}. The proof is based on the flow map representation. Since the activation function is positively two-homogeneous in the network parameters, the flow field $-\nabla(\delta_\pi \Risk)$ is positively one-homogeneous, which means that half-rays move as half-rays and cones are preserved under the flow. Both results also apply directly to explicitly regularized risk functions with suitable homogeneity such as
\[
F_\eps (\pi) = \Risk(\pi) + \eps\,\int_{\overline\Theta} |a|^2+|w|^2 + b^2\,\pi(\d a\otimes \d w \otimes \d b), \qquad\eps\geq0.
\]
The proofs are given in the appendix.

\subsection{Growth of Second Moments}
Recall that we had denoted the second moment of $\pi$ by
\[
N(\pi) := \int_{\overline \Theta} |\theta|^2\,\pi(\d\theta) = \int_{\overline\Theta}a^2 + |w|_{\ell^2}^2 + b^2\,\pi(\d a\otimes \d w \otimes \d b).
\]
Under gradient flow training, $N$ can only grow sublinearly in time.

\begin{lemma}\label{lemma sublinear growth}
If $\pi_t$ evolves by the Wasserstein-gradient flow of $\Risk$, then
\[
N(\pi_t) \leq 2 [N(\pi_0) + \Risk(\pi_0)\,t]\qquad\text{and }\quad
\lim_{t\to \infty} \frac{N(\pi_t)}t = 0.
\]
\end{lemma}

The proof of the Lemma is fairly general and does not require the specific structure of $\phi$, $\ell$ and $\P$. The second moments of $\pi_t$ control the path-norm (or {\em Barron norm}) $\|f_{\pi_t}\|_{\mathcal B}$, where
\[
\|f\|_{\mathcal B} = \inf\left\{\int_{\overline\Theta} |a|\,\big[|w|+|b|\big]\,\pi(\d a\otimes \d w \otimes \d b)\:\bigg|\:\pi \in \mathcal P_2\text{ and }f_\pi = f \:\overline\P-\text{a.e.}\right\}.
\]
Neural networks with low Barron norm are poor approximators for general Lipschitz functions in the $L^2(\P)$-topology if the data distribution $\overline\P$ is truly high-dimensional \cite{approximationarticle}. The slow growth of the Barron norm implies that gradient flows cannot decrease $L^2$-population risk at rates faster than $t^{-4/(d-2)}$ for general target functions $f^*$ \cite{dynamic_cod}.

On the other hand, the key insight of \cite{chizat2018global} is that the homogeneity of $\phi$ implies that if the velocity potentials $\delta_\pi\Risk $ were to converge to a non-trivial limit as $t\to\infty$, it would lead to exponential growth of $N(\pi_t)$.

\begin{lemma}\label{lemma cone argument}
Assume that
\begin{align*}
\delta_\pi\Risk(\pi_{t},\cdot)  &= \int (\partial_1\ell)\big(f_{\pi_{t}}(x),y\big)\, \phi(\cdot,x) \,\P(\d x \otimes \d y)\\
\nabla_\theta\,\delta_\pi\Risk(\pi_{t},\cdot)  &= \int (\partial_1\ell)\big(f_{\pi_{t}}(x),y\big)\,\nabla_\theta \phi(\cdot,x) \,\P(\d x \otimes \d y)
\end{align*}
converge to a function $g$ and a vector field $V = \nabla g$ respectively, locally uniformly on $\overline\Theta$. If $g\not\equiv 0$, then
\[
N(\pi_t) \geq c_0\,e^{c_1t}
\]
for some $c_0, c_1>0$ and $t\gg1$.
\end{lemma}

We give both proofs in the appendix for the reader's convenience. The proof of Lemma \ref{lemma cone argument} is the only point in the document where the assumptions \ref{assumption omnidirectional} and \ref{assumption Sard} are used.

\subsection{Main Results: Lipschitz Loss}
We can now characterize the convergence of Wasserstein gradient flows for the risk functional $\Risk$ with omni-directional initial conditions entirely. We consider the {\em simultaneous $\omega$-limit set}
\[
\Omega^{lim}: = \{(g,V)\:\big|\:\exists\ t_n\to \infty\text{ s.t. }\delta_\pi\Risk (\pi_{t_n},\cdot)\to g, \: \nabla(\delta_\pi\Risk )(\pi_{t_n},\cdot)\to V\text{ locally uniformly}\}.
\]

\begin{theorem}\label{theorem main}
Assume the above conditions \ref{assumption continuity l}, \ref{growth condition on l}, \ref{assumption convexity l}, \ref{assumption P Borel}, \ref{assumption P finite moments}, \ref{assumption density in L1}, \ref{assumption RL1}, \ref{assumption growth at infinity}, \ref{assumption finite risk}, \ref{assumption initial cone}, \ref{assumption omnidirectional}, \ref{assumption Sard}. Then
\begin{enumerate}
\item $\Omega^{lim}$ is not empty and
\item $\lim_{t\to\infty} \Risk(\pi_t) = \inf_{\pi}\Risk(\pi)$ if and only if $\Omega^{lim}$ consists of only one element.
\end{enumerate}
If  $\Omega^{lim}$ consists of only one element, then that element is $\lim_{t\to\infty}(\delta_\pi \Risk)(\pi_t;\theta) \equiv 0$.
\end{theorem}

The first statement of the Theorem does not require \ref{assumption omnidirectional} and \ref{assumption Sard}. Verifying that there exists only a single element in $\Omega^{lim}$ is non-trivial. Generally uniqueness is proved by showing that the limit satisfies an equation which can be studied separately. Even assuming that there exists a limiting measure $\pi_\infty$ such that $\pi_t\to \pi_\infty$, the natural `zero dissipation in the limit' condition
\[
0 = \int_{\overline\Theta}|V|^2\,\pi_\infty(\d\theta) = \int_{\overline \Theta} \big|\nabla (\delta_\pi\Risk)\big|^2(\pi_\infty,\theta)\,\pi_\infty(\d\theta)
\]
only shows that $g(\theta) = 0$ for $\pi_\infty$-almost all $\theta$. This is significantly weaker than {\em every} $\theta$ if $\pi_\infty$ concentrates on a small set. Functions satisfying the zero-dissipation property are stationary points of the flow, and there are many of them which are not the global minimum. Not assuming the existence of a limit $\pi_\infty$, linearization becomes highly involved since the map $\pi\mapsto f_\pi$ is far from injective. This makes the question of uniqueness non-trivial. Despite these obstacles, we believe the result to be a first step towards a general convergence theory. Without assumptions \ref{assumption omnidirectional}, \ref{assumption Sard}, a weaker result holds.

\begin{theorem}\label{theorem main weak}
Assume that \ref{assumption continuity l}, \ref{growth condition on l}, \ref{assumption convexity l}, \ref{assumption P Borel}, \ref{assumption P finite moments}, \ref{assumption density in L1}, \ref{assumption RL1}, \ref{assumption growth at infinity}, \ref{assumption finite risk} and \ref{assumption initial cone} hold. Then
\begin{enumerate}
\item $\Omega^{lim}$ is not empty and
\item $\lim_{t\to\infty} \Risk(\pi_t) = \inf_{\pi}\Risk(\pi)$ if and only if $|\Omega^{lim}|$ contains only the function $g=0$.
\end{enumerate}
\end{theorem}

Having found a characterization for risk converging to zero for omni-directional initial conditions, we now turn our attention to the question of whether we can find a minimizer of risk. The corollary is almost immediate.

\begin{corollary}\label{corollary minimizer 1}
If in addition to the assumptions of Theorem \ref{theorem main} we assume that there exist a sequence of times $t_n\to\infty$ and a measure $\pi_\infty$ on $\Theta$ such that $\pi_{t_n} \to \pi_\infty$ in 2-Wasserstein distance and $|\Omega^{lim}|=1$, then $f_{\pi_\infty} = f^*$, i.e.\ $\pi_\infty$ is a risk minimizing measure. 
\end{corollary}

The existence of a suitable subsequence is guaranteed if the $p$-th moments of $\pi_t$ remain uniformly bounded for some $p>2$. A version of Corollary \ref{corollary minimizer 1} also holds under the same assumptions ad Theorem \ref{theorem main weak}. Under a weaker condition, we obtain a weaker statement.

\begin{corollary}\label{corollary minimizer 2}
If in addition to the assumptions of Theorem \ref{theorem main} we assume that the second moments $N(\pi_t)$ remain uniformly bounded (at least along a subsequence of times $t_n\to\infty$) and $|\Omega^{lim}|=1$, there exists a measure $\pi_\infty\in\mathcal P$ such that $f_{\pi_\infty}=f^*$ (i.e.\ a risk minimizing measure). 
\end{corollary}

Again, there is a version of Corollary \ref{corollary minimizer 2} under the conditions of Theorem \ref{theorem main weak}. The measures $\pi_{t_n}$ may not have a limit in 2-Wasserstein distance, but we can explicitly construct $\pi_\infty$ from $\pi_{t_n}$ by a reparametrization argument. If the second moments of $\pi_t$ remain uniformly bounded in time, there exists a subsequence $\pi_{t_n}$ which converges to a limit in $p$-Wasserstein distance for all $p<2$. This is not sufficient since the map $\pi\mapsto f_\pi$ is not continuous in this topology, as the following example shows:
\[
\pi_n := \frac1{n}\,\delta_{\theta = \sqrt{n}e_1} + \left(1-\frac1n\right)\,\delta_{\theta=0}, \quad \pi_\infty = \delta_{\theta=0}, \qquad f_{\pi_n} \equiv \phi(e_1,\cdot)\neq 0 = f_{\pi_\infty}.
\]

\subsection{Main Results: Smooth Loss and Bounded Data}\label{section main general loss}

The previous results can be used to prove a different version of the main theorems in which a regularity assumption is shifted from the loss function to the data distribution. We make the following observation: If $\overline\P$ is supported on $B_R(0)$ for some $R>0$, then
\[
\|f_{\pi_t}\|_{L^\infty(\overline\P)} \leq (1+R)\,N(\pi_t) \leq 2\,(1+R)\big[N(\pi_0)+ \Risk(\pi_0)\,t\big]
\]
by Lemma \ref{lemma sublinear growth}. 
Assume further that $|y|\leq R$ $\P$-almost surely and that the loss function $\ell$ satisfies
\[
\sup_{|y|\leq S, \:|y'|\leq R} \big[|\ell| + |\partial_1\ell|\big](y,y') < \infty
\]
for any $S<\infty$. We consider the modified loss function
\[
\ell_S(y, y') = \begin{cases} \ell(y,y') & |y|\leq S\\
	\ell(S, y') + (\partial_1\ell)(S,y')\cdot (y-S) & y>S\\
	\ell(-S,y') + (\partial_1\ell)(-S,y')\cdot (y+S) & y<-S
	\end{cases}
\]
for $S>0$ and the modified risk functional
\[
\Risk_S(\pi) = \int_{\R^d\times\R} \ell_S(f_\pi(x),y)\,\P(\d x\otimes \d y).
\]
Then
\begin{enumerate}
\item $\ell_S$ satisfies all conditions previously imposed and
\item $\Risk_S(\pi) = \Risk(\pi)$ for all $\pi$ such that $N(\pi)\leq\frac{S}{1+R}$.
\end{enumerate}
In particular, the gradient flow $\pi_t^S$ of $\Risk_S$ starting at a fixed measure $\pi_0$ independent of $S$ exists and $\pi_t^S = \pi_t^{S'}$ describes the gradient flow of $\Risk$ for $t< {c\,\min\{S,S'\}}/{(1+R)}$. Note that the growth rate of $N$ does not depend on $S$. Thus we may take $S$ to infinity to find that the gradient flow $\pi_t$ of $\Risk$ exists on $[0,\infty)$. We have shown the following.

\begin{lemma}
Replace assumption \ref{assumption P finite moments} by the stronger assumption
\begin{enumerate}[label=(P\arabic*'), ref=(P\arabic*')]\setcounter{enumi}{1}
\item\label{assumption P compact support} $\P$ has compact support in $\R^{d+1}$.
\end{enumerate}
and condition \ref{growth condition on l} by the weaker assumption that 
\begin{enumerate}[label=(L\arabic*'), ref=(L\arabic*')]\setcounter{enumi}{1}
\item\label{weaker growth condition on l} $\ell$ satisfies the growth condition
\[
\sup_{|y|\leq S, \:(x,y') \in\, \spt(\P)} \big[|\ell| + |\partial_1\ell|\big](y,y') < \infty
\]
for all $S>0$.
\end{enumerate}
Then the gradient flow $\pi_t$ of $\Risk$ starting at $\pi_0\in \mathcal P_2(\overline\Theta)$ exists. If $\pi_0$ is omni-directional, so is $\pi_t$ for all $t>0$.
\end{lemma}

Also the results on the convergence of gradient flows generalize under slightly stronger assumptions on the loss function. Examining the proof of Theorem \ref{theorem main}, we find that the relevant properties of the proof are the following, which we postulate as conditions for smooth loss functions.

\begin{enumerate}[label=(SL\arabic*), ref=(SL\arabic*)]\setcounter{enumi}{0}
\item\label{SL Lavrentiev gap} $\inf_\pi \Risk(\pi) = \widetilde \Risk(f^*)$.
\item\label{SL integrability} For any $\pi \in \mathcal P, \theta\in \Theta$, the functions $(\partial_1\ell)(f_\pi(x), y)\,f^*(x)$ and $(\partial_1\ell)(f_\pi(x), y)\,\phi(\theta;x)$ are $\P$-integrable.
\item\label{SL density} If  
\[
\int_{\R^d\times \R} (\partial_1\ell)(f_\pi(x), y)\,\phi(\theta; x)\,\P(\d x\otimes \d y) \to 0\quad\forall\theta\in \Theta
\]
{then also}
\[
\int_{\R^d\times \R} (\partial_1\ell)(f_\pi(x), y)\,f^*(x)\,\P(\d x\otimes \d y) \to 0
\]
\end{enumerate}

The smooth loss conditions have interpretations as follows. \ref{SL Lavrentiev gap} shows that there exist measures $\pi\in \P$ which have finite loss and that $f^*$ can be suitably approximated by $f_\pi$ in risk, i.e.\ there exists no `Lavrentiev gap' for this risk functional. \ref{SL integrability} is achieved in the previous case by the fact that $\partial_1\ell \in L^\infty$ and $f^*, \phi(\theta, \cdot)$ in $L^1$. In the examples below, this will be extended to more general $L^p/L^{p'}$ dual pairings. \ref{SL density} is a density result for the functions $\phi(\theta,\cdot)$ similar to \ref{SL Lavrentiev gap}, but for dual pairings instead of energies.

\begin{theorem}\label{theorem main 2}
Assume the conditions \ref{assumption continuity l}, \ref{weaker growth condition on l}, \ref{assumption convexity l}, \ref{assumption P Borel}, \ref{assumption P compact support}, \ref{assumption density in L1}, \ref{assumption RL1}, \ref{assumption growth at infinity}, \ref{assumption finite risk}, \ref{assumption initial cone}, \ref{assumption omnidirectional}, \ref{assumption Sard}, \ref{SL Lavrentiev gap}, \ref{SL integrability} and \ref{SL density}. Then
\begin{enumerate}
\item $\Omega^{lim}$ is not empty and
\item $\lim_{t\to\infty} \Risk(\pi_t) = \inf_{\pi}\Risk(\pi)$ if and only if $\Omega^{lim}$ consists of only one element.
\end{enumerate}
If  $\Omega^{lim}$ consists of only one element, then that element is $\lim_{t\to\infty}(\delta_\pi \Risk)(\pi_t;\theta) \equiv 0$.
\end{theorem}

Corollaries \ref{corollary minimizer 1} and \ref{corollary minimizer 1} remain true also in this case, and Theorem \ref{theorem main weak} generalizes in the same way.

\section{Examples}\label{section examples}

We imposed a number of abstract conditions on the loss function, data distribution, and initial condition. In this section we consider concrete situations where the conditions are met.

\begin{example}[Huber-type loss functions]
Loss functions like Huber loss
\[
\ell_H(y, y') = \begin{cases}\frac12\,|y-y'|^2 &\text{if } |y-y'|<1\\|y-y'|-\frac12 &\text{else}\end{cases}
\]
and pseudo-Huber loss 
\[
\ell_{psH}(y,y') = \sqrt{|y-y'|^2+1}-1
\]
are Lipschitz-continuous and have Lipschitz continuous first derivatives. They grow at infinity for any fixed $y'\in \R $, so for distributions $\P$ with finite first moments, they are well-defined and coercive. Pseudo-Huber loss is strictly convex, so $f^*$ is uniquely defined. Huber loss is only strictly convex around its minimum and not strictly convex away from it, so if 
\[
\P^x = \delta_{y=-2} + \delta_{y=2},
\]
then any $f^*(x) \in [-1,1]$ is an admissible minimizer. More generally, we can observe that Huber loss has a characteristic smoothing length scale. If data is very spotty on a larger length scale, uniqueness of the minimum may not hold.

Huber-type loss functions interpolate between finding the mean (quadratic loss) and median (linear loss) of the conditional distribution $\P^x$. They are less sensitive to outliers than mean squared loss. 
\end{example}

\begin{example}[Binary classification]
Consider the soft-plus loss function
\[
\ell(y,y') = \log\left( 1+ \exp\big(-yy'\big)\right).
\]
Clearly $\ell$ is Lipschitz-continuous and grows linearly if $yy' \to -\infty$. The loss decays to zero exponentially fast on the other hand if $yy' \to \infty$.
In the case of binary classification, the data distribution $\P$ is given by a measure of data samples $\bar \P$ and marginals $\P^x = \lambda(x)\,\delta_{\{y=1\}} + (1-\lambda(x))\,\delta_{\{y=-1\}}$ for $\lambda(x)\in[0,1]$, thus
\[
f^*(x) = \argmin_{\alpha \in[-\infty,\infty]} \lambda(x)\,\log(1+e^{-\alpha}) + (1-\lambda(x))\,\log(1+e^{\alpha}).
\]
If there exists a minimal uncertainty $0<\lambda_{min}\leq \lambda(x)$ in the prediction of labels for $\overline\P$-almost every $x$, then there is a unique minimizing $\alpha\in\R$ which satisfies
\begin{align*}
0 &=\frac{d}{d\alpha} \big[\lambda\,\log(1+e^{-\alpha}) + (1-\lambda)\,\log(1+e^{\alpha})\big]\\
	&= \frac{-\lambda\,e^{-\alpha}}{1+e^{-\alpha}} + (1-\lambda)\frac{e^\alpha}{1+e^\alpha}\\
	&= (1-\lambda - \lambda\,e^{-\alpha})\,\frac{e^\alpha}{1+e^\alpha}.
\end{align*}
Hence 
\[
e^{-\alpha} = \frac{1-\lambda}\lambda, \qquad\text{i.e.}\quad f^*(x) = \alpha(\lambda(x)) = \log\left(\frac{\lambda(x)}{1-\lambda(x)}\right)
\]
is (essentially) uniformly bounded in the presence of minimal uncertainty and continuous in $x$ if $\lambda$ is, while for perfectly predictable labels (i.e.\ when $\lambda\in\{0,1\}$ $\overline\P$-almost surely)
\[
f^* = \infty \cdot \big(1_{\{\lambda = 1\}} - 1_{\{\lambda=0\}}\big)= (2\lambda-1)\cdot\infty
\]
is almost surely not finite-valued. Our results apply directly in the first situation, but not the second. In this case, the minimizer is simple enough to be understood directly. Namely,
\begin{align*}
\Risk(\pi_t)\to 0 &\LRa \log\big(1+\exp\big(-y\,f_{\pi_t}(x)\big)\big) \to 0 & \P-\text{almost everywhere}\\
	&\LRa y\,f_{\pi_t}(x)\to \infty &\P-\text{almost everywhere}\\
	&\LRa \frac{\exp(-y\,f_{\pi_{t}}(x))}{1+\exp(-yf_{\pi_t}(x))}\to 0 & \P-\text{almost everywhere}\\
	&\LRa \delta_\pi \Risk(\theta) \to 0 \qquad \forall\ \theta\in \Theta
\end{align*}
\end{example}

\begin{remark}
If $\P = \frac1N\sum_{j=1}^N \delta_{(x_j,y_j)}$ is an empirical measure with $x_i\neq x_j$ for $i\neq j$, we can achieve perfect classification. The decay of loss only characterizes the behavior $f_\pi \to f^*$ at data points. A more careful analysis shows that classifiers converge to maximum margin solutions \cite{Chizat:2020aa}, which also characterizes the behavior in between data points to be, in some sense, optimal. This corresponds to the observation in the proof or Theorem \ref{theorem main} that if the gradient of risk is uniformly small and the excess risk is large, then the second moments $N$ are decreasing.

The model considered in \cite{Chizat:2020aa} is neural network-like with the same homogeneity as ReLU-activated two layer networks. Our results for ReLU activation cannot be applied directly since $\overline\P$ must not be an empirical measure. Even assuming existence, the proof of Lemma \ref{lemma cone argument} hinges on the analysis of the flow map and requires smoothness. However, very localized measures of the form
\[
\P = \sum_{j=1}^N \eps^{-(d+1)}\,\eta\left(\frac{\cdot - (x_j, y_j)}\eps\right) \cdot \mathcal{L}^{d+1}
\]
would be expected to lead to similar behavior. Here $\eta$ is any compactly supported probability density, $\eps>0$ is a small parameter and $\mathcal L$ denotes Lebesgue measure. Empirical measures are recovered in the singular limit $\eps\to0$.
\end{remark}

\begin{example}
Huber-type and softplus loss allow for data distributions $\P$ with bounded first moments. Under the (much stronger) assumption that the support of the data distribution $\P$ is compact, algebraic loss functions
\[
\ell_p(y,y') = \frac{|y-y'|^p}p, \qquad p\in (1,\infty)
\]
are admissible. Conditions \ref{SL Lavrentiev gap} and \ref{SL density} are met since functions of the form $f_\pi$ are uniformly dense in the space of continuous functions on compact sets (Universal Approximation Theorem) and continuous functions are dense in Lebesgue spaces $L^q(\overline\P)$ for any Radon probability measure on $\R^d$ due \cite[Theorem 2.11]{fonseca2007modern}.

\ref{SL integrability} is met since $f^*\in L^p(\overline \P)$ due to the boundedness assumption on the support of $\P$ and the growth condition on $\ell_p$. Observe furthermore that $f_\pi, \phi(\theta,\cdot) \in L^\infty(\P)$ for fixed $\pi \in \mathcal P$.
\end{example}

\begin{example}[Data distributions]
Using Theorem \ref{theorem admissible data}, many data distributions of practical importance are admissible for Lipschitz loss, including all distributions which have a bounded density with respect to Lebesgue measure and decay suitably fast at infinity, e.g.\ 
\begin{enumerate}
\item distributions with continuous density on a bounded open set or
\item Gaussian mixture models with uniformly bounded means and variances.
\end{enumerate}
For general loss functions, the first class is admissible, while Gaussian mixture models are excluded for purely technical reasons.
\end{example}

The question which lower-dimensional objects have admissible geometry remains open at this point. 

\begin{example}[Initial condition]
Initial conditions need to be omni-directional and satisfy the scaling property 
\[
|a|^2 \leq |w|_{\ell^2}^2 + |b^2|.
\]
The easiest choice of admissible condition is to independently choose $(w,b)$ uniformly distributed on $S^d$ and $a$ uniformly distributed on $[-1,1]$. In applications, a popular choice is to initialize $(w,b)$ according to a standard normal distribution with mean zero and unit variance. It is easily computed that
\[
\E\big[|w|^2 + b^2] = d+1
\] 
and due to \cite[Theorem 3.1.1]{vershynin2018high}, the norm of $(w,b)$ {\em concentrates} close to $\sqrt d$ in the sense that
\[
\P_\N\big(\big\{ (w,b) : \big|\sqrt{|w|^2+b^2} - \sqrt{d}\big| > \eps\big\}\big) \leq 2\exp\left(-c\,d\,\eps\right)
\]
for a universal constant $c>0$. Thus if $a$ is distributed on a domain sufficiently smaller than $O(\sqrt{d})$ and $d$ is reasonably large with respect to the network width, then with high probability $\pi_0$ satisfies \ref{assumption initial cone}.
\end{example}

\section{Conclusion}\label{section conclusion}

We have shown that the convergence of gradient descent training with infinitesimal step size for two-layer networks with ReLU or leaky ReLU activation starting at omni-directional initial conditions is equivalent to the convergence of the velocity potential to a unique limit (under certain technical conditions). The result holds for a fairly general class of loss functions and data distributions. Convergence along subsequences is guaranteed by compactness.

A number of questions remain open.

\begin{itemize}
\item We have shown that convergence to minimal Bayes risk is equivalent to the convergence of the velocity potentials $\delta_\pi\Risk $ to a unique limit (for suitable initial conditions). Whether the limiting potential is generally unique remains one of the most relevant open questions in theoretical machine learning.

\item To prove existence and make use of the flow map representation, we are restricted to population risk for suitable data distributions. Especially for the case of low-dimensional data in high-dimensional spaces, the regularity condition on the data manifold is hard to understand and check.

Even if a data distribution $\P$ is admissible and the gradient flows of empirical risk functionals sampled from $\P$ exist, it is not obvious if they approach the gradient flow of population risk. Convergence of gradient flows is known to hold under certain (usually hard to check) conditions \cite{serfaty2011gamma,sandier2004gamma}, but there are non-convergence results in applications of practical importance in other fields \cite{mielke2012emergence, dondl2019effect}.

\item It remains open whether a Morse-Sard type property like \ref{assumption Sard} generally holds in higher dimension (possibly under additional assumptions), or whether it can be eliminated from the proof of Lemma \ref{lemma cone argument}. The property is guaranteed to hold by the Morse-Sard theorem for sub-analytic functions \cite{bolte2006nonsmooth} if the measure $\P$ is a finite sum of point-masses. Unfortunately, highly concentrated measures are at odds with the regularity assumption \ref{assumption RL1} which is required in the ReLU setting.

\item Even if a risk-minimizing measure $\pi$ exists and risk decays to its minimum, it is unclear whether 
\begin{enumerate}
\item $\Risk(\pi_t)$ decays at a rate,
\item the second moments $N(\pi_t)$ remain bounded, and 
\item the measures $\pi_t$ converge to a minimizer weakly or in $2$-Wasserstein distance (assuming that $N(\pi_t)$ remains bounded).
\end{enumerate}

\item State-of-the-art neural network architectures can have hundreds or even thousands of layers, far from the two-layer situation considered here. In \cite{chizat2018global}, the authors consider also the case of smooth bounded activation functions which are linear in one direction (and sufficiently smooth). These results apply to network architectures in which every node in the outermost layer is given its own set of parameters for deeper layers. Other models for mean field training of deep networks \cite{araujo2019mean, nguyen2019mean, nguyen2020rigorous, sirignano2019mean} are very different from models for shallow networks. No analogous result is available in this setting.
\end{itemize}

\appendix

\section{Proofs}\label{appendix proofs}

In this section we collect the previously omitted proofs. 

\subsection{Proofs from Section \ref{section background}}
We begin with a proof of Theorem \ref{theorem admissible data} on admissible data distributions for our analysis. Recall that we say that \ref{assumption RL1} holds for $\overline\P$ if the map
\[
S^d\to L^1\big(\sqrt{1+|x|^2}\cdot\overline\P\big), \qquad (w,b) \mapsto 1_{\big\{x|w^Tx+b>0\big\}}
\]
is Lipschitz continuous.

\begin{proof}[Proof of Theorem \ref{theorem admissible data}]
Consider an open half-space $H$ in $\R^d$. Since $\partial H$ is convex, there exists a closest point $x_H\in \partial H$ to the origin in $\R^d$. If $0\notin \partial H$, we can characterize 
\[
H = \begin{cases} \{x\in\R^d\:|\: \langle x, x_H\rangle > |x_H|^2\} &0\notin H\\
\{x\in\R^d\:|\: \langle x, x_H\rangle < |x_H|^2\}  &0\in H\end{cases}.
\]
Thus $H$ is represented by as $H = \{x|w^Tx+b>0\}$ for
\[
(w,b)\in \R^d,\qquad (w,b) = \lambda\, (\pm x_H, |x_H|^2), \qquad\lambda> 0.
\]
In particular, there is a canonical representative in the unit sphere
\[
(w,b) = \left(\pm \frac{x_H}{\sqrt{|x_H|^2 + |x_H|^4}}, \:\frac{|x_H|^2}{\sqrt{|x_H|^2 + |x_H|^4}}\right).
\]
For the derivative estimates below, we can assume without loss of generality that $b\notin\{0,\pm1\}$, as the Lipschitz estimate extends to the boundary points by uniform continuity. Given a density $\rho$, denote
\[
\tilde \rho(x) = \rho(x)\,\sqrt{1+|x|^2},\qquad
\tilde \rho^+(x) = \limsup_{y\to x}\tilde\rho(y)
\]
and observe that $\tilde\rho^+$ satisfies the same decay estimate as $\sqrt{1+|x|^2}\cdot\rho$ and is integrable with respect to Lebesgue measure. $\tilde\rho^+$ is upper semi-continuous.

We prove the first claim. Let $(w,b)\in S^d$. Without loss of generality, $w = \sqrt{1-b^2}\,e_1$. Here, we can even take $\eps =0$ in the decay condition and compute
\begin{align*}
\limsup_{h\to 0^+} &\bigg\|\frac{1_{\{x|w^Tx+b+h>0\}} - 1_{\{x|w^Tx+b>0\}}}{h} \bigg\|_{L^1(\sqrt{1+|x|^2}\cdot\overline\P)}\\
	&= \limsup_{h\to 0^+} \frac1h\int_{\big\{x\big|-(b+h)<w^Tx<-b\big\}} \,\tilde\rho(x)\dx\\
	&= \limsup_{h\to 0^+} \frac1h\int_0^{\frac h{\sqrt{1-b^2}}}\int_{\{x_1=-b/\sqrt{1-b^2}\}} \,\tilde\rho(x-te_1)\,\H^{d-1}(\d x)\,\dt\\
	&= \frac1{\sqrt{1-b^2}}\limsup_{h\to 0} \int_{0}^1 \int_{\{x_1=-b/\sqrt{b}\}}\tilde\rho\left(x-s\,\frac{h}{\sqrt{1-b^2}}e_1\right)\,\H^{d-1}(\d x)\,\ds\\
	&\leq \frac1{\sqrt{1-b^2}}\int_{\{x_1=-b/\sqrt{1-b^2}\}} \tilde\rho^+(x)\,\H^{d-1}(\d x)\\
	&\leq \frac1{\sqrt{1-b^2}}\int_{\R^{d-1}} \frac{C}{\big(1 + |x|^2 + \frac{b^2}{1-b^2}\big)^\frac{d+1}2}\,\dx\showlabel\label{eq b=0}\\
	&= \frac{\sqrt{1-b^2}}{b^2} \int_{\R^{d-1}} \frac{C}{\left(\frac{1-b^2}{b^2} + \left|\frac{\sqrt{1-b^2}\,x}{b}\right|^2 + 1\right)^{\frac{d+1}2}}\,\left(\frac{\sqrt{1-b^2}}b\right)^d\dx\\
	&\leq \frac{\sqrt{1-b^2}}{b^{2}} \int_{\R^{d-1}} \frac{1}{\big(1+|y|^2\big)^\frac{d+1}2}\,\dy\showlabel\label{eq b>0}
\end{align*}
This is bounded from above uniformly when $b$ is bounded away from $\pm 1$ by \eqref{eq b=0} and close to $b=\pm 1$ by \eqref{eq b>0}. 
 The passage to the limit can be justified using Fatou's lemma. A similar estimate holds for the limit $h\to0^-$, and the derivative in direction of increasing/decreasing $w$ radially is the same as that in direction of changing $b$.

Now consider $v\in \R^d$ such that $\langle v,w\rangle = 0$. Taking the derivative in $v$, the half-spaces are no longer aligned and rather than a slab, their symmetric difference is the union of two wedges. Again without loss of generality, we assume that $w= \sqrt{1-b^2}\,e_1$ and $v= e_2$. Then
\begin{align*}
\big\{x &: \big|1_{\{w^Tx+b>0\}} - 1_{\{(w+hv)^Tx + b>0\}}\big|=1\big\} \\
	&= \{w^Tx+b>0\}\, \Delta\, \{(w+hv)^Tx + b>0\}\\
	&= \left\{x: \sqrt{1-b^2}\,x_1 + hx_2 < -b < \sqrt{1-b^2}\,x_1 \right\} \cup \left\{x: \sqrt{1-b^2}\,x_1 < -b < \sqrt{1-b^2}\,x_1 + hx_2 \right\} \\
	&= \left\{x: - \frac{b}{\sqrt{1-b^2}} < x_1 < -\frac{b +hx_2}{\sqrt{1-b^2}}\right\} \cup \left\{x:  -\frac{b+hx^2}{\sqrt{1-b^2}} < x_1 < -\frac{b}{\sqrt{1-b^2}}\right\}\\
	&\subset\left\{x: \frac{-b - |hx_2|}{\sqrt{1-b^2}} < x_1 < \frac{-b + |hx_2|}{\sqrt{1-b^2}}\right\}
\end{align*}
Then we can therefore bound the tangential derivative as follows.
\begin{align*}
\limsup_{h\to 0^+}& \frac{\|1_{\left\{x\:|\: (w+h v)^Tx+b>0\right\}} - 1_{\{x|w^Tx+b>0\}}\|_{L^1(\sqrt{1 + |x|^2}\cdot \overline\P)}}{h}\\
	&\leq \limsup_{h\to0}\frac1{|h|} \int_{\R^{d-1}} \int_{\frac{-b - |hx_2|}{\sqrt{1-b^2}}}^\frac{-b + |hx_2|}{\sqrt{1-b^2}} \tilde\rho(y_1,x_2,\dots,x_d)\,\dy_1\dx_2\dots\dx_d\\
	&= \limsup_{h\to 0} \int_{\R^{d-1}} \int_{-1}^1 \frac{|x|_2}{\sqrt{1-b^2}}\,\tilde\rho\left(-\frac b{\sqrt{1-b^2}} + \frac{th|x_2|}{\sqrt{1-b^2}}, x_2,\dots,x_d\right)\d t\,\dx_2\dots\dx_d\\
	&\leq \frac{1}{\sqrt{1-b^2}}\int_{\R^{d-1}}|x_2|\,\tilde\rho^+\left(-\frac{b}{\sqrt{1-b^2}}, x_2,\dots,x_d\right)\,\dx_2\dots\dx_d.
\end{align*}
The remainder of the proof proceeds as above, except for the weight of $|x_2|$ in front of the density. While previously a decay as $|x|^{-(d+1+\eps)}$ for $\eps>0$ would have been sufficient, here we need decay as $|x|^{-(d+2+\eps)}$ at infinity.

The proof of the second statement is similar to the first. The third statement is immediate from the definition. To consider the fourth statement, let $\M$ be the set of Radon probability measures
\[
\M = \left\{\mathbb P\::\: \big\|1_{\{w^T\cdot + b>0\}} - 1_{\{\tilde w^T \cdot + \tilde b>0\}}\big\|_{L^1(\sqrt{1+|x|^2}\cdot \overline\P)} \leq L\big[|w-\tilde w| + |b-\tilde b|\big]\right\}.
\]
Assume that $\P_n$ is a sequence in $\M$ and $\P$ is a probability measure such that $\P_n\to \P$ in the sense of Radon measures. Then
\[
\P(U) \leq \liminf_{n\to\infty}\P_n(U)
\]
for all open sets $U$. In particular, the symmetric differences of open sets are open. This establishes the Lipschitz condition by the previous analysis since
\[
\|1_H - 1_{H'} \|_{L^1(\sqrt{1+|x|^2}\cdot \overline\P_n)} = \big(\sqrt{1+|x|^2}\cdot\overline\P_n\big)(H\Delta H') \leq L \big[|w-w'| + |b-b'|\big]
\]
where $H, H'$ are the half-spaces 
\[
H = \{x: w^Tx+b>0\}, \qquad H' = \{x:(w')^Tx+b'>0\}.
\]
\end{proof}

In particular, sums of Gaussian distributions or compactly supported regular distributions as they occur in density estimation are admissible data distributions for our purposes. They can be computed from a given finite data sample and mollify the problem sufficiently for our convergence result.

Many full-dimensional data distributions satisfy \ref{assumption RL1}, but distributions with a bounded density on the hypersphere is admissible. If data is concentrated on a manifold of dimension $k$, the intuition is that $k$ should not have any `straight' segments which lie mostly in a lower-dimensional affine subspace.

\begin{remark}
It is easy to see that no measure $\overline\P$ such that $\overline\P(\partial H)>0$ for any half-space $H$ satisfies \ref{assumption RL1}. In fact, the map $(w,b)\mapsto 1_{\{w^Tx+b>0\}}$ is not even continuous for such distributions. This in particular excludes empirical measures, but also lower-dimensional data manifolds which are entirely contained in an affine subspace.
\end{remark}

A complete characterization of admissible measures is beyond the scope of this article. We move on to the proof of the measurability of $f^*$.

\begin{proof}[Proof of Lemma \ref{lemma measurable}]
{\bf First claim.} Note that for every $\alpha\in\R$, the map $x\mapsto L_x(\alpha)$ is $\overline\P$-measurable by the construction of the marginal measures in \cite[Theorem 4.2.4]{attouch2014variational}. Furthermore, $L_x(\alpha)<\infty$ for $\overline\P$-almost every $x$ and all $\alpha\in\R$ since $L_x(0)$ is integrable and $L$ is Lipschitz continuous in $\alpha$ (with the same Lipschitz constant as $\ell$).

Hence the function $(x,\alpha)\mapsto L_x(\alpha)$ is a Caratheodory integrand (finite, continuous in $\alpha$ and measurable in $x$) where $\alpha$ belongs to a second countable complete space. It is well-known that such functions are jointly measurable, see \cite[Theorem 14.75]{aliprantisinfinite}. {\em Sketch of proof:} Define $U = \{x : L_x(0) <\infty\}$ and the sequence of functions
\[
\Psi_k:\R\times U\to [0,\infty), \qquad \Psi_k(\alpha, x) = L_x\left(k\cdot \left\lfloor \frac\alpha k\right\rfloor\right)
\]
which are product measurable. Since $L_x$ is continuous in $\alpha$, $L_x(\alpha) = \lim_{k\to\infty} \Psi_k(\alpha,x)$ and thus $(x,\alpha)\mapsto L_x(\alpha)$ is product measurable on $\R\times U$. Since $\overline\P(\R^d\setminus U) = 0$, we find that the map is product measurable on the entire space. Note that we establish product-measurability with respect to the Borel $\sigma$-algebra on $\R$ and the $\sigma$-algebra on $\R^d$ generated by $\overline\P$, which contains all $\overline\P$-null sets and is in general larger than the Borel $\sigma$-algebra. 

{\bf Second claim.} To find $f^*$ we use the Kuratowski-Ryll-Nardzewski Selection Theorem \cite[Theorem 14.86]{aliprantisinfinite} which states that if the (possibly multi-valued) map
\[
M: U\to 2^\R, \qquad M(x) = \argmin_\alpha L_x(\alpha)
\]
is a weakly measurable correspondence (see \cite[Chapter 14]{aliprantisinfinite}) with nonempty closed values in a Polish space, then it admits a measurable selector. The only non-trivial fact is the weak measurability of $M$, which means that we need to check that the set
\[
\{x\in U\:|\:\phi(x)\cap V\neq\emptyset\} = \left\{x\in U\:\bigg| \:\exists\ \alpha \in V\text{ s.t.\ }L_x(\alpha) = \inf_{\alpha'\in\R}L_x(\alpha')\right\}
\]
is measurable in $U$ whenever $V$ is open. Any open set $V$ admits a countable dense subset $\widehat V$, which means that
\[
m_V:U\to[0,\infty), \quad m(x) = \inf_{\alpha\in V} L_x(\alpha) = \inf_{\alpha\in \widehat V} L_x(\alpha)
\]
is measurable due to the continuity of $L_x$ in $\alpha$. We conclude that
\[
\{x\in U\:|\:M(x)\cap V\neq\emptyset\} = \{x\in U\:|\: m_V(x) = m_\R(x)\}
\]
is measurable.
 \end{proof}
  
For the reader's convenience, we link Wasserstein gradient flows to classical gradient flows.

\begin{proof}[Proof from Section \ref{section wasserstein gf}]
Let $\Theta = (\theta_1,\dots,\theta_m)$ and $f_\Theta(x) = \frac1m\sum_{i=1}^m \phi(\theta_i; x)$. Then 
\begin{align*}
\nabla_{\theta_i}\Risk(\Theta) &= \nabla_{\theta_i} \int_\Theta \ell\big(f_\Theta(x), y\big)\,\P(\d x)\\
	&= \int_\Theta (\partial_1\ell)\big(f_\Theta(x), y\big)\,\nabla_{\theta_i}\left(\frac1m \sum_{i=1}^m \phi(\theta_i; x)\right)\,\P(\d x)\\
	&= \frac1m \int_\Theta (\partial_1\ell)\big(f_\Theta(x), y\big)\,(\nabla_\theta\phi)(\theta_i;x)\,\P(\dx)\\
	&= \frac1m \nabla (\delta_\pi\Risk)\left(\frac1m \sum_{j=1}^m\delta_{\theta_j}; \theta_i\right).
\end{align*}
Thus if the parameters $\theta_i$ evolve by the law $\dot \theta_i = - m\,\nabla_{\theta_i}\Risk(\Theta)$ for all $i=1,\dots,m$, then their distribution $\pi^m:= \frac1m \sum_{i=1}^m\delta_{\theta_i}$ satisfies the transport equation
\[
\dot\pi^m_t = \div\big(\pi^m_t\,\nabla(\delta_\pi\Risk)(\pi^m_t; \cdot)\big)
\]
by the flow map representation. This is precisely the PDE formulation of the 2-Wasserstein gradient flow. 
\end{proof}
  
Now we show that gradient flow of $\Risk$ exists for any initial condition $\pi_0\in \mathcal P_2(\overline\Theta)$ and that the omni-directionality of the initial measure is preserved along the gradient flow evolution (for finite time). Except for technical issues stemming from the lack of regularity in ReLU activation, the analysis follows \cite[Appendix B]{chizat2018global}.

For technical purposes, it is necessary to consider ReLU activation on the whole space and not just the (non-convex) cone $\Theta$. The natural extension which is linear in $a$ faces the aforementioned issues of non-differentiability. We therefore define $\phi$ on $\R^{d+2}$ as
\[
\phi(a,w,b ; x) = \eta\left(\frac{a^2 - |w|^2 - b^2}{a^2 + |w|^2 + b^2}\right)\,a\,\sigma\left({w^Tx+b}\right)
\]
where 
\[
\eta(z) = \begin{cases} 1 &z\leq 0\\ 0 & z\geq \frac12\end{cases}, \qquad \eta' \leq 0
\]
is a smooth cut-off function. In particular, note that
\[
\phi(a,w,b; x) = a\,\sigma(w^Tx+b)
\]
for all $(a,w,b)\in \Theta$.

\begin{lemma}\label{lemma extension of phi}
Consider $\phi:\R^{d+2}\times \R^d\to\R$ as above. Then
\begin{enumerate}
\item For any $\rho\in L^\infty(\P)$, the vector field
\[
V_\rho:\R^{d+2}\to \R, \quad V_\rho(\theta) = \int_{\R^d}\rho(x,y)\,\nabla_\theta\phi(\theta,x)\,\P(\d x\otimes \d y) 
\]
is Lipschitz continuous with Lipschitz constant $\|\rho\|_{L^\infty}$.

\item For any $\theta\in \partial\Theta$ and any $\rho\in L^\infty(\P)$, $V_\rho$ is tangent to $\partial\Theta$ at $\theta$. 
\end{enumerate}
\end{lemma}

\begin{proof}
{\bf Lipschitz-regularity.} 
We observe that
\begin{align*}
\nabla_\theta \phi(\theta;x) &= \eta'\left(\frac{a^2 - |w|^2 - b^2}{a^2 + |w|^2 + b^2}\right)a\,\sigma(w^Tx+b)\\
	&\hspace{3cm}\left\{\frac{2}{a^2 + |w|^2 + b^2}\,\begin{pmatrix}a\\ -w\\ -b\end{pmatrix} -2 \frac{a^2 - |w|^2 - b^2}{\big[a^2 + |w|^2 + b^2\big]^2}\begin{pmatrix} a\\ w\\ b\end{pmatrix}\right\}\\
	&\hspace{1.5cm} + \eta\left(\frac{a^2 - |w|^2 - b^2}{a^2 + |w|^2 + b^2}\right)\begin{pmatrix} \sigma(w^Tx+b)\\ a\,1_{\{w^Tx+b>0\}}x\\ a\,1_{\{w^tx+b>0\}}\end{pmatrix}\\
	&=  \sigma(w^Tx + b)\,V_1(a,w,b) + \eta\left(\frac{a^2 - |w|^2 - b^2}{a^2 + |w|^2 + b^2}\right)\,a\,1_{\{w^Tx+b>0\}}\begin{pmatrix} 0\\ x\\ 1\end{pmatrix}
\end{align*}
where $V_1$ is a positively $0$-homogeneous vector field which is smooth on the sphere $a^2 + |w|^2 + b^2$. Thus for any $x\in \R^d$, the product $\sigma(w^Tx+b)\cdot V_1$ is Lipschitz-continuous on the sphere with Lipschitz-constant $\leq C\sqrt{1+|x|^2}$ and positively one-homogeneous. The triangle-inequality implies that the term is $C\sqrt{1+|x|^2}$-Lipschitz on the whole space. Since the first moments of $\P$ are finite, we deduce that
\[
(a,w,b) \mapsto \int_{\R^d\times \R} \rho(x,y)\,\sigma(w^Tx + b)\,V_1(a,w,b) \,\P(\d x\otimes \d y)
\]
is $C\,\|\rho\|_{L^\infty(\P)}$-Lipschitz continuous. Estimating the second term requires higher regularity of $\P$ as expressed in condition \ref{assumption RL1}. We note that if $\eta>0$, then
\[
\frac{a^2 - |w|^2 - b^2}{a^2+ |w|^2 + b^2} \leq \frac12 \qquad\Ra\quad a^2 - |w|^2 - b^2 \leq \frac{a^2 + |w|^2 + b^2}2\qquad\Ra\quad a^2 \leq 3\,\big[|w|^2 + b^2\big].
\]
For any fixed $(a,w,b)$, $(\tilde a, \tilde w, \tilde b)$, denote by $\eta, \tilde\eta$ the cut-off function evaluated at the respective point. Without loss of generality, we assume that $|(\tilde w, \tilde b)|\leq |(w,b)|$ and compute that
\begin{align*}
\bigg| \int_{\R^d\times \R} &\rho(x,y) \eta\,a\,1_{\{w^Tx+b>0\}} \begin{pmatrix} 0\\ x\\ 1\end{pmatrix} \,\P(\d x\otimes \d y) - \int_{\R^d\times \R} \rho(x,y) \tilde\eta\,\tilde a\,1_{\{\tilde w^Tx+\tilde b>0\}} \begin{pmatrix} 0\\ x\\ 1\end{pmatrix} \,\P(\d x\otimes \d y)\bigg|\\
	&\leq \|\rho\|_{L^\infty(\P)}\int_{\R^d\times \R} \big|\eta a\,1_{\{w^Tx+b>0\}} - \tilde\eta\tilde a\,1_{\{\tilde w^Tx+\tilde b>0\}}\big|\cdot\sqrt{1+|x|^2}\,\P(\d x\otimes \d y)\\
	&\leq \|\rho\|_{L^\infty}\int_{\R^d} \big\{|\eta a - \tilde \eta \tilde a|\,1_{\{w^Tx+b>0\}} + |\tilde \eta \tilde a|\,\big|1_{\{w^Tx+b>0\}}  - 1_{\{w^Tx+b>0\}} \big|\big\}\sqrt{1+|x|^2}\,\overline\P(\d x)\\
	&\leq \|\rho\|_{L^\infty}\left\{|\eta a - \tilde \eta \tilde a|\int_{\R^d}\sqrt{1+|x|^2}\,\overline\P(\d x) + C |\tilde \eta \tilde a|\,\left|\frac{(w,b)}{\sqrt{|w|^2+b^2}} - \frac{(\tilde w, \tilde b)}{\sqrt{|\tilde w|^2 + \tilde b^2}}\right|\right\}\\
	&\leq C\,\|\rho\|_{L^\infty} \left[|\eta a - \tilde \eta \tilde a| + \left| \frac{|(\tilde w, \tilde b)|}{|(w,b)|} (w,b) - (\tilde w, \tilde b)\right|\right]
\end{align*}
The first term on the right is bounded by
\[
|\eta a - \tilde \eta \tilde a| \leq C\,\|(a, w, b) - (\tilde a, \tilde w, \tilde b)\|
\]
since $(a,w,b) \mapsto a\,\eta(a,w,b)$ is a positiively one-homogeneous function which is Lipschitz-continuous on the sphere and thus Lipschitz continuous. The second term on the right is
\[
\left| \frac{|(\tilde w, \tilde b)|}{|(w,b)|} (w,b) - (\tilde w, \tilde b)\right| = \big|P^{S^{d+1}_r}(w,b) - P^{S^{d+1}_r}(\tilde w, \tilde b)\big|
\]
where $P^{S^{d+1}_r}$ denotes the projection onto the sphere of radius $r = |(\tilde w, \tilde b)|$ around the origin. Since both $(w,b)$ and $(\tilde w, \tilde b)$ lie on the exterior of the sphere, this map is one-Lipschitz.

{\bf Tangency condition.} We can write $\Theta = \{(a,w,b)\:|\: -a^2 + |w|^2+ b^2 >0\}$. Then the normal to $\partial\Theta$ is parallel to  $\nabla (-a^2+|w|^2 + b^2) = 2(-a, w, b)$. Fix $(a,w,b)$ and consider any $x$ such that $w^Tx+b\neq 0$. Since $\eta \equiv 1$ close to $\partial\Theta$, we have
\begin{align*}
2|\theta|\,\left\langle \nabla_\theta \phi(\theta;x), \nu_{\partial\Theta}\right\rangle &= \begin{pmatrix} \sigma(w^Tx+b)\\ a\,1_{\{w^Tx+b>0\}}x\\ a\,1_{\{w^tx+b>0\}}\end{pmatrix} \cdot \begin{pmatrix}-a\\ w\\ b\end{pmatrix}\\
	&= -a\,\sigma(w^Tx+b) + a\,\sigma'(w^Tx+b)\,(w^Tx+b)\\
	&=0
\end{align*}
since $\sigma$ is positively one-homogeneous. The equality also holds trivially at $x$ for which $w^Tx+b = 0$, so in particular after integration.
\end{proof}

In the calculus of variations (which encompasses the study of gradient flows), different notions of convexity play a key role. In vector spaces, convexity is a condition along straight lines, which (at least for Hilbert spaces) are the same as length-minimizing curves. The natural generalization to (geodesically complete) metric spaces is to consider the notion of convexity where a functional is `convex' if it is convex along constant-speed length minimizing geodesics. The analogy is particularly strong in $2$-Wasserstein space, which carries the formal structure of a Hilbert manifold, see \cite{otto2001geometry} or \cite[Chapter 15]{villani2008optimal}. This concept of convexity is referred to as {\em displacement convexity} and has been recognized since \cite{mccann1997convexity} as a useful notion when considering gradient flows.

The functionals we consider are not convex in Wasserstein space -- in fact, convergence to a global minimizer is not guaranteed. For the existence of gradient flows, a weaker concept suffices. Recall that the functional $\Risk$ is called $\lambda$-displacement convex, if the following holds: If $\pi_t$ is a geodesic in Wasserstein space, then
\[
\Risk(\pi_t) \leq t\,\Risk(\pi_0) + (1-t)\,\Risk(\pi_1) + \frac\lambda2\,t(1-t)\,W_2^2(\pi_0, \pi_1).
\]
By analogy with the smooth Euclidean case, we can think of the condition as a lower bound on the Hessian $D^2\Risk \geq - \lambda I$. Convexity corresponds to $0$-convexity and uniform convexity to $\lambda$-convexity for $\lambda<0$.

\begin{proof}[Proof of Lemma \ref{lemma existence}]
{\bf Step 1.} We first show that $\Risk$ is $\lambda$-displacement convex for a suitable $\lambda\in\R$ depending on $\overline\P, \ell$ and $\eta$. Unlike Wasserstein space $\mathcal P_2(\R^{d+2})$, the space of measures on the cone $\P_2(\overline\Theta)$ not geodesically convex due to the non-convexity of $\Theta$. We therefore use Lemma \ref{lemma extension of phi} to extend $\phi$ to the whole space $\R^{d+2}$. Without the extension, Wasserstein space is not geodesically convex and we cannot explicitly characterize geodesics. We will later show that if $\spt(\pi_0)\subseteq \overline\Theta$, then $\spt(\pi_t) \subseteq \overline\Theta$ for all $t\geq 0$.

Let $\pi_0, \pi_1 \in \mathcal P_2(\R^{d+2})$. Denote by $\pi_t$, $t\in[0,1]$ any unit speed geodesic between $\pi_0$ and $\pi_1$, i.e.\
\[
\pi_t = \big[t\,P^\theta + (1-t)\,P^{\theta'}\big]_\sharp\gamma
\]
where $P^\theta, P^{\theta'}$ are the projection from $\R^{d+2}\times \R^{d+2}$ to the first and second components respectively, $F_\sharp\mu$ denotes the push-forward of the measure $\mu$ along the map $F$, and $\gamma$ is an optimal transport plan between $\pi_0$ and $\pi_1$. 

It suffices to show that $h(t) := \Risk(\pi_t)$ is $\lambda\,W_2^2(\pi_0, \pi_1)$-convex on $[0,1]$, i.e.\ $h''\geq -\lambda\,W^2_2(\pi_0,\pi_1)$. We prove a stronger statement, namely that $h'$ is $\lambda\,W_2^2(\pi_0,\pi_1)$-Lipschitz, which can be thought of as a type of Hessian bound from both sides instead of just one side. Note that
\begin{align*}
\frac{d}{dt}f_{\pi_t}(x) &= \frac d{dt} \int_{\R^{d+2}\times\R^{d+2}} \phi(t\theta + (1-t)\theta'; x)\,\gamma(\d\theta\otimes\d\theta')\\
	&=\int_{\R^{d+2}\times\R^{d+2}} \big\langle (\nabla_\theta\phi)(t\theta + (1-t)\theta'; x), \:\theta-\theta'\big\rangle \,\gamma(\d\theta\otimes\d\theta')\\
h'(t) &= \frac{d}{dt} \Risk(\pi_t) \\
	&= \frac{d}{dt} \int_{\R^d\times \R } \ell \big(f_{\pi_t}(x), y\big)\,\P(\d x \otimes\d y)\\
	&= \int_{\R^d\times \R } (\partial_1\ell) \big(f_{\pi_t}(x), y\big)\,\frac{d}{dt}f_{\pi_t}(x)\,\P(\d x \otimes\d y)\\
	&= \int_{\big(\R^{d+2}\big)^2} \left\langle \int_{\R^d\times \R } (\partial_1\ell) \big(f_{\pi_t}, y\big) (\nabla_\theta\phi)(t\theta + (1-t)\theta'; x)\,\P(\d x\otimes \d y), \:\theta-\theta'\right\rangle\gamma(\d\theta\otimes\d\theta').
\end{align*}
Due to Lemma \ref{lemma extension of phi}, the map
\[
\theta\mapsto \int_{\R^d\times \R } (\partial_1\ell)(f_\pi(x),y)\,(\nabla_\theta\phi)(\theta,x)\,\P(\d x\otimes \d y)
\]
is $L\,\|\partial_1\ell\|_{L^\infty(\P)}$-Lipschitz, and $\|\partial_1\ell\|_{L^\infty}$ is bounded uniformly by the Lipschitz-condition on the loss function. Consequently
\begin{align*}
|h'(t_1) - h'(t_0)| &\leq C \int \big|t_1\theta + (1-t_1)\theta' - t_0\theta - (1-t_0)\theta'\big|\,|\theta-\theta'|\,\gamma(\d\theta\otimes \d\theta')\\
	&= C\,|t_1-t_0|\,W_2^2(\pi,\pi').
\end{align*}
It follows that $\Risk$ is in fact $\lambda$-displacement convex for a universal constant $\lambda$ which only depends on $\phi, \ell$. 

{\bf Step 2.} Existence of the gradient flow follows directly from \cite[Theorem 11.2.1]{ambrosio2008gradient}.

{\bf Step 3.} In this step, we show that the gradient flow of $\Risk$ preserves the cone $\Theta$, i.e.\ if $\pi_0\in \mathcal{P}_2(\overline\Theta)$, then $\pi_t\in \mathcal{P}_2(\overline\Theta)$ for all $t\geq 0$. 
Note that the flow field
\[
(t, \theta)\mapsto -\nabla(\delta_\pi\Risk ) \big(\pi_t; \theta)
\] 
is Lipschitz-continuous in $\theta$ with a uniform Lipschitz constant for all times by Lemma \ref{lemma extension of phi}. Like in Section \ref{section continuity eqn}, we find that $\pi_t = X(t,\cdot)_\sharp \pi_0$ where $X$ is the flow map defined by
\[
\begin{pde}
\frac{d}{dt}X(t,\theta) &= -\nabla(\delta_\pi\Risk ) \big(\pi_t; X(t,\theta)\big) &t>0\\
X(0,\theta) &= \theta
\end{pde}.
\]
It thus suffices to show that $X(t,\theta)\in \Theta$ for every $\theta\in\Theta$. This is immediate since 
\begin{align*}
\dot X(t,\theta) &= -\nabla (\delta_\pi\Risk )(\pi_t; X(t,\theta))\\
	&= -\int_\Theta (\partial_1\ell)\big(f_{\pi_t}(x),y\big)\,(\nabla_\theta\phi)(\theta,x)\,\P(\d x \otimes \d y)
\end{align*}
is parallel to $\partial\Theta$ on the boundary since $\nabla_\theta\phi(\theta,x)$ is parallel to $\partial\Theta$ at $\theta$ whenever it is defined -- see also Section \ref{section initial}.
\end{proof}

We now prove that the gradient flow preserves the omni-directionality of measures (for finite positive time).

\begin{proof}[Proof of Lemma \ref{lemma omni-directional}]
{\bf Preliminary analysis.} Again, we use the flow map $X$. Since $[\nabla (\delta_\pi\Risk)]_{\mathrm{Lip}}\leq C$, we deduce from Lemma \cite[Lemma 4]{ambrosio2008transport} that $X$ satisfies
\[
[X(t,\cdot)]_{\mathrm{Lip}} \leq \exp \left(\int_0^t  C\ds\right) \leq \exp\big(Ct\big).
\]
Since we can solve the ordinary differential equation backwards in time for any $\theta\in \Theta$, $X(t,\cdot):\Theta\to \Theta$ is a bi-Lipschitz homeomorphism. 

Finally, we note that $X(t,\lambda\theta) = \lambda X(t,\theta)$ for all $\lambda, t>0$ since
\[
\frac{d}{dt} \lambda\,X(t,\theta) = -\lambda\,\nabla(\delta_\pi\Risk )(\pi_t; \theta) = - \nabla (\delta_\pi\Risk )(\pi_t; \lambda\theta)
\]
due to the homogeneity of $\phi$. Thus, the flow preserves half-rays and cones.

{\bf Proof of omni-directionality.} Consider an open cone $C\subseteq\Theta$. Then by \cite[Lemma 4]{ambrosio2008transport}, we have 
\[
\pi_t(C) =
 \int_{\Theta}1_C(X(t,\theta))\,\pi_0(\d\theta) = \int_\Theta 1_{X(t,\cdot)^{-1}(C)}(\theta)\,\pi_0(\d\theta) = \pi_0\big(X(t,\cdot)^{-1}(C)\big)>0
\]
since also $X(t,\cdot)^{-1}(C)$ is an open cone in $\Theta$.
\end{proof}

\begin{remark}\label{remark flow map}
The analysis of the flow map shows more. Since rays are preserved, the projected measures $P^{S^{d+1}}_\sharp(\pi_t)$ on the unit sphere $S^{d+1}\cap \Theta$ evolve by the continuity equation
\[
\frac{d}{dt}P^{S^{d+1}}(\pi_t) = \div \left(\left(P^{S^d}_\sharp \pi_t\right)\cdot \nabla^{S^d}(\delta_\pi\Risk )(\pi_t, \cdot)\right)
\]
where $\nabla^{S^{d+1}} f(\theta) = (I-\theta\otimes\theta)\nabla f(\theta)$ is the tangential gradient to the unit sphere. In particular, if $P^{S^{d+1}}_\sharp \pi_0$ has a density $\rho_0$ with respect to the uniform measure on $S^{d+1}\cap \Theta$, then $P^{S^{d+1}}_\sharp \pi_t$ has a density $\rho_t$ and
\begin{align*}
\exp\big(- Ct\big)&\inf_{\theta\in \Theta\cap S^{d+1}} \rho_0(\theta) \leq \inf_{\theta\in \Theta\cap S^{d+1}} \rho_t(\theta)
\leq \sup_{\theta\in \Theta\cap S^{d+1}} \rho_t(\theta) \leq \exp\big(Ct)\big)\sup_{\theta\in \Theta\cap S^{d+1}} \rho_0(\theta).
\end{align*}
\end{remark}

Next we show that the second moment of $\pi_t$ grows at most sublinearly in time. 

\begin{proof}[Proof of Lemma \ref{lemma sublinear growth}]
We compute
\begin{align*}
\frac{d}{dt} N(\pi_t) &= \frac{d}{dt} \int_\Theta |\theta|^2\,\pi_t(\d\theta)\\
	&= -\int_\Theta \langle \nabla|\theta|^2, \nabla_\theta (\delta_\pi\Risk )\rangle \,\pi_t(\d\theta)\\
	&= -2 \int_\Theta \langle \theta, \nabla_\theta(\delta_\pi\Risk )\rangle\,\pi_t(\d\theta)\\
	&\leq 2 \left(\int_\Theta|\theta|^2 \,\pi_t(\d\theta)\right)^\frac12 \left(\int_\Theta|\nabla_\theta(\delta_\pi\Risk )|^2 \,\pi_t(\d\theta)\right)^\frac12\\
	&= 2 N(\pi_t)^{1/2} \left|\frac{d}{dt}\,\Risk(\pi_t)\right|^{1/2},
\end{align*}
i.e.\ 
\[
\frac{d}{dt} \sqrt{N(\pi_t)} \leq \left|\frac{d}{dt}\,\Risk(\pi_t)\right|^{1/2}.
\]
Note that $\pi_t$ has finite second moments, so the quadratic test function $|\cdot|^2$ in the variational formulation is admissible. If $\pi_0$ is compactly supported, then so is $\pi_t$ for all $t>0$ by the flow map representation and the linear growth of the flow field at infinity. In this situation, the identity is obvious. Otherwise, the argument is easily justified by using approximating test functions $\eta(R-|\theta|)\,|\theta|^2$ where $\eta$ is a smooth cutoff function satisfying $\eta'\geq 0$, $\eta(z) = 0$ for $z\leq 0$ and $z=1$ for $z\geq 1$. 
Thus for every $0< T<t$ we have
\begin{align*}
\sqrt{N(\pi_t)} &= \sqrt{N(\pi_T)} + \int_T^t \frac{d}{ds} \sqrt{N(\pi_s)}\ds\\
	&\leq \sqrt{N(\pi_T)} + \int_T^t\left|\frac{d}{ds}\,\Risk(\pi_s)\right|^{1/2}\ds\\
	&\leq \sqrt{N(\pi_T)} + \left(\int_T^t1\ds\right)^\frac12 \left(\int_T^t\left|\frac{d}{ds}\,\Risk(\pi_s)\right|\ds\right)^\frac12\\
	&= \sqrt{N(\pi_T)} + \sqrt{t-T} \,\left[\Risk(\pi_T) - \Risk(\pi_t)\right]^{1/2}
\end{align*}
and therefore
\begin{equation}\label{eq sublinear moment bound}
N(\pi_t) \leq 2 \big[N(\pi_T) + (t-T)\,\left[\Risk(\pi_T) - \Risk(\pi_t)\right]\big].
\end{equation}
Since $\Risk(\pi_t)$ is monotone decreasing and bounded from below (by zero), $\Risk(\pi_t)$ converges to a limit. Thus, for every $\eps>0$ we can choose $T>0$ such that $0 < \Risk(\pi_T) - \Risk(\pi_t) < \eps$ for every $t>T$ and hence
\[
\limsup_{t\to\infty} \frac{N(\pi_t)}{t} \leq 2  \limsup_{t\to \infty} \left[\Risk(\pi_T) - \Risk(\pi_t)\right] \leq 2\eps.
\]
\end{proof}

Note that the proof applies in great generality to models with a linear structure. The next proof concerns the exponential growth of second moments if the velocity potential converges to a non-trivial limit. It is adapted from \cite{chizat2018global} and repeated in this context for the reader's convenience. The following proof is the only point at which the Morse-Sard property is used in this article. It is also the only argument which hinges on the omni-directionality of the initial parameter distribution.

\begin{proof}[Proof of Lemma \ref{lemma cone argument}]
If $g\not\equiv 0$, there exists $(a,w,b)\in \R^{d+2}$ such that $g(a,w,b) \neq 0$. Since $g$ is linear in $a$ and positively two-homogeneous in $(a,w,b)$, there exists $\theta\in \Theta$ such that $g(\theta)\neq 0$. Since $g(-a, w,b) = - g(a,w,b)$, there therefore exist $\eps>0, \theta\in\Theta$ such that $g(\theta)< -2\eps\,|\theta|^2$. Thus there exists an open cone $C$ in $\R^{d+2}$ such that
\begin{enumerate}
\item $C\cap \Theta\neq \emptyset$,
\item $g(\theta) < -\eps\,|\theta|^2$ for all $\theta\in C$, and
\item $\langle \nabla g, \nu_C \rangle > 0$ does not vanish on $\partial C$ where $\nu$ is the inner normal vector to $\partial C$.
\end{enumerate}

Using Assumption \ref{assumption Sard}, $C$ can for example be chosen as
\[
C = \left\{\theta \in \R^{d+2} : \frac{g(\theta)}{|\theta|^2} < -t\right\} 
\]
for some $t\in (\eps, 2\eps)$. We define the localized second moments 
\[
N_C(\pi):= \int_C |\theta|^2\,\pi(\d\theta).
\]
Since $\partial C\cap S^{d+1}$ is compact and $\nabla (\delta_\pi \Risk)\to \nabla g$ locally uniformly, we find that there exists $T_0>0$ such that $\langle \nabla (\delta_\pi\Risk)(\pi_t;\cdot), \nu_C\rangle >0$ on $\partial C$ for all $t>T_0$. In particular, no mass flows out of $C$ after time $T_0$: If $X_\theta(T_0) \in C$ then also $X_\theta(t)\in C$ for $t>T_0$. Thus
\[
N_C(\pi_t) = \int_C |\theta|^2\,X(t)_\sharp \pi_0 = \int_{X(t)^{-1}(C)} \big|X_\theta(t)\big|^2\,\pi_0(\d\theta)\geq \int_{X(T_0)^{-1}(C)}\big|X_\theta(t)\big|^2\,\pi_0(\d\theta).
\]
Secondly since $(\delta_\pi\Risk)(\pi_t,\cdot) \to g$ locally uniformly, there exists $T_1>0$ such that 
\[
(\delta_\pi\Risk)(\pi_t;\theta)\leq -\frac\eps2\,|\theta|^2
\]
for all $t>T_1$. Without loss of generality, we assume that $T_1=T_0$. In particular
\[
\frac{d}{dt}\,|X_\theta(t)|^2 = \langle X_\theta(t), -\nabla (\delta_\pi\Risk)\big(\pi_t; X_\theta(t)\big) = -2\,(\delta_\pi\Risk)\big(\pi_t; X_\theta(t)\big) \geq \eps\,\big|X_\theta(t)\big|^2
\]
for $t>T_0$ and $\theta \in X(T_0)^{-1}(C)$, using the positive two-homogeneity of $(\delta_\pi\Risk)$. Thus 
\[
|X_\theta(t)|^2 \geq |X_\theta(T)|^2\,\,e^{\eps(t-T_0)}
\]
for $t>T_0$ and $\theta \in X(T_0)^{-1}(C)$, and consequently
\begin{align*}
N(\pi_t) &\geq N_C(\pi_t)\\
	& \geq \int_{X(T_0)^{-1}(C)}\big|X_\theta(t)\big|^2\,\pi_0(\d\theta) \\
	&\geq \int_{X(T_0)^{-1}(C)}\big|X_\theta(T_0)\big|^2\,e^{\eps(t-T_0)}\,\pi_0(\d\theta)\\
	&= e^{\eps(t-T_0)}\,N_C(\pi_{T_0}).
\end{align*}
If $\pi_0$ is omni-directional, then so is $\pi_{T_0}$ and $N_C(\pi_{T_0})>0$.
\end{proof}

\begin{remark}
Morally, assumption \ref{assumption Sard} is used to control the sign of a boundary flux term. Without an assumption of this type, the term would at most be asymptotically non-negative. It would be necessary to control the size of the boundary term by a volume contribution and its asymptotic behavior. 
\end{remark}

\subsection{Proofs from Section \ref{section main}}

We use these results to prove the main theorem.

\begin{proof}[Proof of Theorem \ref{theorem main}]
{\bf Step 1.} Since $\delta_\pi \Risk$ and $\nabla\delta_\pi\Risk$ are positively two- and one-homogeneous respectively, we find that $(\delta_\pi\Risk)(\pi, 0) = 0$ and $\nabla (\delta_\pi \Risk)(\pi,0) = 0$ for any $\pi \in \mathcal P_2$. According to Lemma \ref{lemma extension of phi}, we may assume a uniform Lipschitz bound on $\nabla (\delta_\pi \Risk)(\pi_t,\cdot)$. This implies a uniform Lipschitz bound also on $\delta_\pi\Risk$ on bounded sets. We thus conclude that $(\delta_\pi\Risk)(\pi_t,\cdot)$ and $\nabla (\delta_\pi\Risk)(\pi_t,\cdot)$ have convergent subsequences, since Lipschitz-space embeds compactly into the space of continuous functions by the Arzel\`a-Ascoli theorem.

{\bf Step 2.} First, assume that $\Omega^{lim}$ consists of only one element $(g,V)$. Then either $g\not\equiv 0$ or $g\equiv0$. In the first case, we have by Lemma \ref{lemma cone argument} that $N(\pi_t)$ grows exponentially in time, contradicting Lemma \ref{lemma sublinear growth}. In the second case, we use homogeneity to show that
\begin{align*}
-\frac{d}{dt} N(\pi_t) &= \int_\Theta \langle \nabla|\theta|^2, \nabla (\delta_\pi\Risk )\rangle \,\pi_t(\d\theta)\\
	&= 2\int_\Theta \langle \theta, \nabla (\delta_\pi\Risk )\rangle \,\pi_t(\d\theta)\\
	&= 4\int_\Theta \delta_\pi\Risk (\theta)\,\pi_t(\d\theta)\\
	&= 4\int_\Theta \int_{\R^d\times\R } (\partial_1\ell)\big(f_{\pi_t}(x),y\big)\,\phi(\theta,x)\,\P(\d x \otimes \d y)\,\pi_t(\d\theta)\\
	&= 4 \int_{\R^d\times\R } (\partial_1\ell)\big(f_{\pi_t}(x),y\big)\,\int_\Theta \phi(\theta,x)\,\pi_t(\d\theta)\,\P(\d x \otimes \d y)\\
	&= 4\int_{\R^d\times\R }  (\partial_1\ell)\big(f_{\pi_t}(x),y\big)\,f_{\pi_t}(x)\,\P(\d x \otimes \d y)\\
	&= 4\int_{\R^d\times\R } \ell\big(f^*(x),y\big)+  (\partial_1\ell)\big(f_{\pi_t}(x),y\big)\,\big(f_{\pi_t}-f^*\big)(x)\,\P(\d x \otimes \d y)\\
	&\qquad - 4\,\widetilde \Risk (f^*) + 4\int_{\R^d\times\R } (\partial_1\ell)\big(f_{\pi_t}(x),y\big)\,f^*(x)\,\P(\d x \otimes \d y)\\
	&\geq 4\int_{\R^d\times\R } \ell\big(f_{\pi_t}(x),y\big)\P(\d x \otimes \d y) - 4\,\widetilde \Risk (f^*)\\
	&\qquad + 4\int_{\R^d\times\R }(\partial_1\ell)\big(f_{\pi_t}(x),y\big)\,f^*(x)\,\P(\d x \otimes \d y)\\
	&= 4\left[ \widetilde \Risk (f_{\pi_t}) - \widetilde \Risk (f^*)\right] + 4\int_{\R^d\times\R } (\partial_1\ell)\big(f_{\pi_t}(x),y\big)\,f^*(x)\,\P(\d x \otimes \d y).
\end{align*}
due to the first-order convexity condition for $\ell$ in the first argument. The integrals converge since $\partial_1\ell\in L^\infty(\overline\P)$ and $f^*\in L^1(\overline\P)$ by Assumption \ref{assumption minimizer L1}. We know that 
\[
\int_{\R^d\times\R } (\partial_1\ell)\big(f_{\pi_t}(x),y\big)\,\phi(\theta,x)\,\P(\d x \otimes \d y)\to g(\theta) = 0\qquad\forall\ \theta\in\Theta.
\]
Using assumption \ref{assumption density in L1}, we conclude that
\[
\int_{\R^d\times\R} (\partial_1\ell)\big(f_{\pi_t}(x),y\big)\,f^*(x)\,\P(\d x \otimes \d y)\to 0
\]
as $t\to\infty$ since the span of $\{\phi(\theta,\cdot):\theta\in\Theta\}$ is dense in $L^1(\P)$. Thus
\[
\liminf_{t\to\infty} \frac{d}{dt} N(\pi_t) \leq 4\lim_{t\to\infty}\left[ \widetilde \Risk (f_{\pi_t}) - \widetilde \Risk (f^*)\right]\leq 0
\]
because $\widetilde\Risk(f^*) = \inf_{\pi} \Risk(\pi)$. Since $N$ is bounded from below by zero, it cannot decrease linearly at a fixed non-zero rate for all large arguments. We conclude that 
\[
\lim_{t\to\infty} \Risk(\pi_t) = \widetilde\Risk(f^*) = \inf_{\pi} \Risk(\pi).
\]
{\bf Step 3.} Now, assume that $\lim_{t\to\infty} \Risk(\pi_t) = \widetilde \Risk (f^*)$, i.e.\
\[
\lim_{t\to\infty} \int_{\R^d} L_x(f_{\pi_t}(x)) - L_x(f^*(x))\,\overline\P(\d x) = 0
\]
where $L$ is the {\em augmented loss function} discussed in Section \ref{section risk functional}.
Since the first integrand is always larger than the second one, their difference is positive and thus we conclude that
\[
L_x(f_{\pi_t}(x)) - L_x(f^*(x))\to 0 \qquad\text{in }L^1(\overline\P).
\]
At least for a subsequence $t_n\to \infty$, this means that the convergence holds pointwise almost everywhere. Since $L_x(f_{\pi_{t_n}}(x)) \to L_x(f^*(x))$ $\overline\P$-almost everywhere and $L_x(\alpha)\to\infty$ if $|\alpha|\to\infty$, the minimizing sets $M_x$ are compact and that $\dist(f_{\pi_{t_n}}(x), M_x)\to 0$. The pointwise limit of the functions is measurable, and thus without loss of generality $f^*(x) = \lim_{n\to\infty} f_{\pi_{t_n}}(x)$. By the continuity of $\partial_1\ell$, we conclude that
\[
(\partial_1\ell)\big(f_{\pi_{t_n}}(x),y\big) \to (\partial_1\ell)\big(f^*(x),y\big) 
\]
almost everywhere. We recall that
\[
L_x'(f^*(x)) = \int_\R (\partial_1\ell)\big(f^*(x),y\big)\,\P^x(\d y) = 0
\]
for $\overline\P$-almost all $x$ and thus in particular
\[
\int_{\R^d\times\R } (\partial_1\ell)\big(f^*(x),y\big)\,\phi(\theta,x)\,\P(\d x \otimes \d y) = 0\qquad\forall\ \theta\in \Theta.
\]
Since $\phi(\theta,\cdot) \in L^1(\overline\P)$ for any $\theta\in \overline\Theta$, we conclude from the dominated convergence theorem that
\[
\int_{\R^d\times\R } (\partial_1\ell)\big(f_{\pi_{t_n}}(x),y\big)\,\phi(\theta,x)\,\P(\d x \otimes \d y) \to \int_{\R^d\times\R } (\partial_1\ell)\big(f^*(x),y\big)\,\phi(\theta,x)\,\P(\d x \otimes \d y) =0
\]
using the bound
\[
(\partial_1\ell)\big(f_{\pi_{t_n}}(x),y\big)\,\phi(\theta,\cdot) \leq \|\partial_1\ell\|_{L^\infty}\,|\phi(\theta,\cdot)|.
\] 
Thus $\lim_{n\to\infty} \delta_\pi \Risk(\pi_{t_n},\cdot)\equiv 0$. By compactness, we know that $\Omega^{lim}$ is non-empty, and we have showed that for any sequence, we can extract a subsequence for which $g=0$. Since locally uniform convergence is generated by a topology, this means that $g$ is the only possible limit point. The same argument can be used for the gradient.
\end{proof}

\begin{remark}
Note that omni-directionality is only used to exclude the case that $g\not\equiv 0$, whereas $g\equiv 0$ is only an admissible limit if $\Risk(\pi_t)$ decays to MBR. This corresponds to the fact that $\delta_\pi\Risk(\pi,\cdot)\equiv 0$ if and only if $\pi$ is a global minimizer of $\Risk$, see \cite[Proposition 3.1]{chizat2018global}.
\end{remark}

\begin{remark}
Omni-directionality of the initial condition \ref{assumption omnidirectional} and the Morse-Sard property are both involved only in excluding a unique limit $g\not\equiv 0$. They could therefore be replaced for example by the zero-limit assumption
\begin{itemize}[label=(ZL), ref=(ZL)]\setcounter{enumi}{0}
\item If $(\delta_\pi\Risk)(\pi_t,\cdot) \to g$ locally uniformly on $\overline\Theta$ and $g\not\equiv 0$, then
\[
\limsup_{t\to \infty} \frac{N(\pi_t)}t >0.
\]
\end{itemize}
\end{remark}

The following two proofs establish the corollaries to the main theorem concerning minimizers.

\begin{proof}[Proof of Corollary \ref{corollary minimizer 1}]
If we assume in addition that there exists a probability measure $\pi_\infty$ and a sequence of times $t_n\to \infty$ such that $\lim_{n\to\infty} W_2(\pi_{t_n},\pi_\infty) = 0$, then we find by \cite[Theorem 6.9]{villani2008optimal} that
\[
f_{\pi_\infty} = \lim_{n\to\infty}f_{\pi_{t_n}} = f^*.
\]
Thus $\pi_\infty$ minimizes $\Risk$. 
\end{proof}

\begin{proof}[Proof of Corollary \ref{corollary minimizer 2}]
Consider the measures $\mu_t$ on $S^{d+1}\cap\Theta$ defined by
\[
\int_{S^{d+1}\cap \Theta} g(\theta)\,\mu_t(\d\theta) = \int_{\Theta} g\left(\frac\theta{|\theta|}\right)\,|\theta|^2\,\pi_t(\d\theta)
\]
for $g\in C(\overline{S^{d+1}\cap\Theta})$, or equivalently 
\[
\mu_t = P^{S^{d+1}}_{\sharp}\left(|\theta|^2\cdot\pi_t\right)
\]
where $P^{S^{d+1}}:\Theta\to S^{d+1}\cap\Theta$ is the canonical projection. While $P^{S^{d+1}}$ is undefined at $\theta=0$, the projection of the measure with weight $|\theta|^2$ is well-defined. 
Under the moment bound assumption, the measures $\mu_t$ are uniformly bounded and
\[
\int_{S^{d+1}\cap\Theta} \phi(\theta,x)\,\mu_t(\d\theta) = \int_{\Theta} \phi(\theta,x)\,\pi_t(\d\theta) = f_{\pi_t}(x)
\]
for all $x\in U$. By the compactness theorem of Radon measures, there exists a finite Radon measure $\mu_\infty$ on $\overline{\Theta}\cap S^{d+1}$ and a sequence of times $t_n\to\infty$ such that $\lim_{n\to\infty}\tilde\pi_{t_n}= \mu_\infty$. On the closed cone $\overline\Theta$, we can apply the weak convergence of Radon measures to conclude that
\[
\int_{S^{d+1}\cap\overline\Theta} \phi(\theta,x)\,\mu_\infty(\d\theta) = \lim_{n\to\infty} \int_{S^{d+1}\cap\Theta} \phi(\theta,x)\,\mu_{t_n}(\d\theta) = f^*(x)
\]
$\overline\P$-almost everywhere. The measure $\mu_\infty$ is non-negative since all $\mu_{t_n}$ are. We now distinguish two cases:

\begin{enumerate}
\item $\mu_\infty = 0$. In this case $f^*\equiv 0$ and $\pi_\infty:= \delta_{\theta = 0}$ is a risk minimizer.
\item $\mu_\infty\neq 0$. In this case, we define the dilation map $D = \sqrt{\mu_\infty(\Theta)}\,I$ and 
\[
\pi_\infty := \frac1{\mu_\infty(\Theta)}\,D_\sharp \mu_\infty.
\]
Then by homogeneity
\[
f_{\pi_\infty} (x)= \int_\Theta \phi(\theta,x)\d\pi_\infty = \frac1{\tilde\pi_\infty(\Theta)}\int_{S^{d+1}\cap\Theta}\phi\left(\sqrt{\tilde\pi_\infty(\Theta)}\,\theta,x\right)\,\mu_\infty(\d\theta) = f^*(x),
\]
for all $x\in U$, so $\pi_\infty$ is a risk minimizer.
\end{enumerate}
\end{proof}

\begin{remark}
Considering the example
\[
\pi_n := \frac1{n}\,\delta_{\theta = \sqrt{n}e_1} + \left(1-\frac1n\right)\,\delta_{\theta=0}, \quad \pi_\infty = \delta_{\theta=0}, \qquad f_{\pi_n} \equiv \phi(e_1,\cdot)\neq 0 = f_{\pi_\infty}.
\]
after Corollary \ref{corollary minimizer 2}, we observe that 
\[
\mu_n:= P^{S^{d+1}}_\sharp\big(|\theta|^2\cdot\pi_n\big)  \equiv \delta_{\theta=e_1}.
\]
Thus the projection to the unit sphere adds compactness beyond the moment bound.
\end{remark}

\section{The Morse-Sard Property}\label{appendix sard}

The classical Morse-Sard theorem states that the set of critical values of a function $f\in C^k(\R^n,\R^m)$ has zero Lebesgue measure if $k\geq \max\{n-m+1, 1\}$, where $y\in \R^m$ is a critical value of $f$ if there exists $x\in f^{-1}(y)$ such that $Df(x)$ does not have full rank. Combined with the regular value theorem, it is a powerful tool in smooth approximation.

The result is due to Morse for $m=1$ \cite{morse1939behavior} and Sard for general $m\geq 1$ \cite{sard1942measure}. It is easy to extend the result to sufficiently differentiable manifolds, and there are more precise statements available for the Hausdorff measure of $f(S_\nu)$ where
\[
S_\nu = \{x\in \R^n\:|\: \mathrm{rank}(Df_x)\leq \nu\}
\]
for $0\leq\nu\leq n-1$, see \cite[Theorem 3.4.3]{federer66geometric}. In the classical setting, the differentiability assumptions are almost sharp, and the condition can only be weakened to $C^{k-1,1}$ in place of $C^k$ \cite{bates1993toward}. A classical example of Whitney \cite{whitney1935function} which has been generalized to higher dimension e.g.\ in \cite{hajlasz2003whitney} shows that there exists $f\in C^1(\R^2,\R)$ such that $f$ is non-constant on a connected component of its set of critical points, meaning that there is an open interval of critical points.

Morse-Sard theorems are known to fail in infinite dimension even for infinitely smooth maps, unless additional assumptions are imposed. Under weak conditions, however, the set of functions for which the Theorem holds is dense in the $C^0$-topology \cite{eells1968approximate}.

Due to its fundamental importance, some effort has been made to establish Morse-Sard type properties in other function classes. Among these are

\begin{itemize}
\item Morse-Sard theorems in classes of weakly differentiable functions \cite{figalli2008simple,bourgain2010morse, bourgain2015morse, de2001morse,korobkov2018trace}. Here the relation between differentiability and integrability may even be chosen low enough to ensure continuity, but not classical differentiability of the functions under consideration.
\item Morse-Sard theorems for the distance from a submanifold \cite{rifford2004morse} or more generally Lipschitz functions which are given as suprema of smooth functions over suitable index sets \cite{barbet2016sard}.
\item Morse-Sard theorems for subanalytic functions, see \cite{bolte2006nonsmooth}.
\item Morse-Sard theorems for dc functions in two dimensions \cite{pavlica2006morse}. A function is dc if it can be written as differences of two convex functions. In particular, every $C^2$-function with bounded second derivatives is dc. Thus the result cannot be generalized to higher dimensions.
\end{itemize}

In non-smooth function classes, a notion of gradient almost everywhere (with respect to a suitable Hausdorff measure) or a sub-differential is used. 

We show below that \ref{assumption Sard} holds unconditionally in dimension $d=2$. 
In \cite{Chizat:2020aa}, the Morse-Sard theorem for subanalytic functions from \cite{bolte2006nonsmooth} has successfully been used to establish a condition of Morse-Sard property in a very similar application. The subanalytic function that the authors consider in \cite{Chizat:2020aa} is a finite sum of ReLU-like terms and the subanalyticity stems from the finiteness of the sum. The number of summands corresponds to the number of data samples in an empirical measure. The approach is therefore incompatible with assumption \ref{assumption RL1} for the ReLU case.

\subsection{Finitely many neurons}

While it does not apply in our situation, we briefly sketch the result and its application in a similar situation.
Consider $f(x) = \sum_{i=1}^m a_i (w_i^Tx+b_i)_+$. Then $f$ is only Lipschitz-smooth and not $C^1$, but has the Morse-Sard property due to its sub-analyticity. 

\begin{definition}
\begin{enumerate}
\item A set $A\subseteq \R^n$ is called semi-analytic if for every $x\in A$ there exists a neighbourhood $U$ of $x$ and a collection of real analytic functions $f_{ij}, g_{ij}$ such that
\[
A\cap U = \bigcup_{i=1}^N\bigcap_{j=1}^{M_i} \{f_{ij}>0, g_{ij}=0\}.
\]
\item A set $A\subseteq \R^n$ is called sub-analytic if for every $x\in A$ there exists a neighbourhood $U$ of $x$, a semi-analytic set $\tilde A$ in $\R^m$ for $m\geq n$ and an affine map $L:\R^m\to\R^n$ such that $A= L(\tilde A)$.

\item A function $f$ is called sub-analytic if its graph $G_f = \{(x,y)\in \R^{n+1}\:|\:y=f(x)\}$ is sub-analytic.
\end{enumerate}
\end{definition}

We recall the following Morse-Sard theorem.

\begin{theorem}\cite[Theorem 13]{bolte2006nonsmooth}
Let $f:\R^n\to\R$ be a continuous sub-analytic function. Then $f$ is constant on connected components of the set of critical points (defined by the sub-differential), and the set of critical values is countable.
\end{theorem}

We show that this applies to finite ReLU networks. 

\begin{lemma}
$f(x) = \sum_{i=1}^m a_i (w_i^Tx+b_i)_+$ is continuous and sub-analytic.
\end{lemma}

\begin{proof}
Continuity is clear. Let $(x,y)\in G_f$. Since $f$ is affine linear on the set $\{x\:|\: w_i^Tx+b_i\neq 0 \:\forall\ i\}$, the graph is a plane and thus sub-analytic here. Now assume that $w_1^Tx+b_1 =0$ and $w_i^Tx+b_i\neq 0$ for all $i>1$. Then, locally after a rotation and translation we have
\[
f(x) = \alpha^Tx + \max\{x_1,0\}.
\]
The graph of this function is sub-analytic since 
\[
G_f = \{x_1>0,\:\alpha^Tx + x_1-y =0\} \cup \{x_1< 0, \:\alpha^Tx-y =0\} \cup \big(\{x_1 =0\}\cap\{y=0\}\big).
\]
The case when more terms $w_i^Tx+b_i$ vanish can be treated similarly, but is somewhat tedious to write out. It is, however, crucial that the sum is finite to ensure that there are at most finitely many sets to be united and intersected.
\end{proof}

\subsection{ReLU Geometry on the Sphere}

We can exploit the special geometry of the problem to reduce the dimension slightly. We can write
\[
g(a,w,b) = a\,h(w,b), \qquad h(w,b) = \int_{\R^{d+1}} \rho(x,y)\,\P(\d x\otimes \dy )
\]
where $\rho$ is an $L^p$-weak limit of $(\partial_1\ell)(f_\pi(x),y)$ for any $1<p<\infty$. In particular, $h$ is positively one-homogeneous and a ReLU Barron function \cite[Inverse Approximation Theorem]{weinan2019lei}.

To simplify notation, from now on $w$ has $d+1$ components and we identify $w$ with $(w,b)$ and $x$ with $(x,1)$. Denote by $e_a = (1,0,\dots,0)\in \R^{d+2}$ the unit vector pointing in the $a$-direction. At $(a,w)\in S^{d+1}$, the tangential part of $\partial_a$ is 
\[
e_a^\parallel:= e_a - \langle e_a,(a,w)\rangle (a,w) = e_a - a\begin{pmatrix} a\\ w\end{pmatrix}.
\]
We compute 
\[
\nabla^S_a g(a,w) = \partial_a g - a \,(a,w)\cdot\nabla g = \frac{g}a - 2a\,g = \left(\frac1a - 2a\right)\,g. 
\]
Thus on the set
\[
\left\{(a,w,b): \frac1a -2a \neq 0\right\} = \left\{(a,w,b) : a^2 \neq 2\right\} = \left\{(a,w,b)\in \Theta : a^2 < |w|^2 + b^2\right\}
\]
the only critical value is (possibly) zero. On the other hand, consider $(a,w)$ in $\partial \Theta\cap S^{d+1}$. We note that
\[
g(\lambda a, \mu w) = \lambda\mu\,g(a,w)\qquad\forall\ \lambda,\mu >0
\]
and
\[
(\lambda a)^2 + |\mu w|^2 = \lambda^2 a^2 + \mu^2 |w|^2= 1 \qquad\text{if}\quad \mu^2= \frac{1- \lambda^2a^2}{|w|^2} = 2-\lambda^2.
\]
We find that
\[
g\big(\lambda a,\, \mu(\lambda)\, w\big) = \lambda\sqrt{2-\lambda^2}\,g(a,w)
\]
is extremal at $\lambda = \pm 1$ since
\[
0 = \frac{d}{d\lambda} \lambda\sqrt{2-\lambda^2} = 2\,\frac{1-\lambda^2}{\sqrt{2-\lambda^2}} \qquad\LRa\quad \lambda =\pm 1.
\]
Thus all level sets of $g$ intersect $\partial\Theta\cap S^{d+1}$ and linearity in $a$ is insufficient to establish the Morse-Sard property. The Morse-Sard property \ref{assumption Sard} holds if and only if the Morse-Sard property holds for the function $h: S^d\to \R$. Since $h\in C^{1,1}_{loc}(\overline\Theta)$ by Lemma \ref{lemma extension of phi}, the Morse-Sard property holds if $d=2$ by \cite[Theorem 1]{bates1993toward}.

\subsection{A mild counterexample}

We show that functions with similar structural properties as $g$ (or $h$) may not satisfy a Morse-Sard property in dimension $d\geq 8$.

\begin{definition}
Let $\Omega\subseteq \R^d$ be open and bounded. $f:\Omega\to\R$ is called a Barron function if there exists $\pi \in \mathcal P_2(\R^{d+2})$ such that 
\[
f(x) = f_\pi(x) = \int_{\R^{d+2}} a\,\sigma(w^Tx+b)\,\pi(\d a \otimes \d w \otimes \d b)
\]
in $\Omega$. 
\end{definition}

The space of Barron functions is discussed in detail in \cite{E:2018ab,weinan2019lei,barron_new} and \cite[Appendix A]{approximationarticle}. The space is named after Andrew Barron who first established that a large class of functions could be represented in such a way. We cite a simplified version of Barron's main theorem.

\begin{lemma} \cite[Proposition 1]{barron1993universal}
Assume that $f:\R^d\to\R$ is such that its Fourier transform $\widehat f$ satisfies
\[
\int_{\R^d} |\xi|\,\big|\widehat f\big|(\xi)\,\d\xi<\infty.
\]
Then $f$ is Barron. 
\end{lemma}

The original proof uses sigmoidal activation functions, but the result for ReLU activation follows immediately since $\mathrm{ReLU}(x) - \mathrm{ReLU}(x-1)$ is sigmoidal. The criterion is classical. A less well know consequence is the following, derived like \cite[Section IX, point 15] {barron1993universal}.

\begin{corollary}
Assume that $f\in H^s(\R^d)$ for $s> \frac{d}2+1$. Then $f$ is a Barron function.
\end{corollary}

\begin{proof}
Using the identity $\widehat{\partial_jf} = i\,\xi_j\widehat f$ and Parseval's identity, we compute
\begin{align*}
\int_{\R^d} |\xi|\,\big|\widehat f\big|(\xi)\,\d\xi &= \int_{\R^d} |\xi|\,\big|\widehat f\big|(\xi)\,\left(1+ |\xi|^{2s}\right)^\frac12 \left(1+ |\xi|^{2s}\right)^{-\frac12}\,\d\xi\\
	&\leq \left(\int_{\R^d} |\widehat f|^2\,\left(1+ |\xi|^{2s}\right)\d\xi\right)^\frac12 \left(\int_{\R^d} \frac{|\xi|^2}{1+|\xi|^{2s}}\d\xi\right)^\frac12
\end{align*}
By \cite[Satz 9.37]{dobrowolski2010angewandte}, the first factor on the right is finite if and only if $f\in H^s(\R^d)$. The second factor is finite if $2s-2 > d$.
\end{proof}

In particular, if $f$ is $C^k$-smooth on a ball $B_r(0)$ for $k> \frac{d}2 +1$, then we can use standard extension theorems to show that $f$ is Barron. In particular, if $d\geq 7$ then $\frac d2+2 < d-1$ and e.g.\ every $C^5$-function on $B_1(0)\subset \R^7$ is Barron. In particular, we have shown the following.

\begin{corollary}
In dimension $d\geq 7$ there exist Barron functions which do not have the Morse-Sard property.
\end{corollary}

Denote by $S^d$ the unit sphere in $\R^{d+1}$. The function 
\[
h:S^d\to \R, \qquad h(w,b) = \int \rho(x,y)\,\P(\d x \otimes \d y)
\]
from the previous section shares many properties with general Barron functions. In a fixed application, the data distribution $\P$ is fixed, so a Sard property would need to be established in a random feature space (for bounded density $\rho$). Across different applications, the data distribution $\P$ may vary within the class satisfying assumptions \ref{assumption P Borel}, \ref{assumption P finite moments}, \ref{assumption density in L1}, \ref{assumption RL1}. While this class does not exhaust the entire Barron space, we note that \ref{assumption RL1} only implies $C^{1,1}$-smoothness and no sufficient regularity to invoke a classical Sard theorem. We thus expect that there are cases of interest in which the Morse-Sard condition \ref{assumption Sard} does not hold.


\newcommand{\etalchar}[1]{$^{#1}$}

\end{document}